\documentclass[secnum,leqno,a4paper]{rims-bessatsu}
\usepackage{amsmath,amssymb,amsopn,color,mathrsfs,array,hhline,epsfig}
\usepackage{changebar}
\usepackage{extarrows}
\reversemarginpar

\usepackage[OT2,T1]{fontenc}
\DeclareSymbolFont{cyrletters}{OT2}{wncyr}{m}{n}
\DeclareMathSymbol{\Sha}{\mathalpha}{cyrletters}{"58} 
\renewcommand\theequation{\thesection.\arabic{equation}}
\makeatletter
\renewcommand{\section}{\@startsection{section}{1}{\z@}%
  {1.5ex \@plus -1.5ex \@minus -1.2ex}%
  {-.5ex \@plus 1.3ex}%
  {\reset@font\large\bfseries}}
\renewcommand{\subsection}{\@startsection{subsection}{2}{\z@}%
  {-3.25ex\@plus -1ex \@minus -.2ex}%
  {1.5ex \@plus .2ex}%
  {\normalfont\normalsize\bfseries}}
\makeatother
\makeatletter
\def\nottsub#1#2{\m@th\ooalign{\hfil{%
    \ifthenelse{\equal{#1}{\scriptstyle}\or\equal{#1}{\scriptscriptstyle}}{%
    \raise.1ex\hbox{$\scriptscriptstyle\backslash$}}{%
    \raise.16ex\hbox{$\scriptstyle\backslash$}}}\hfil\crcr$#1#2$}}%
\makeatother
\makeatletter
\renewcommand\@pnumwidth{1.55em}
\renewcommand\@tocrmarg{2.55em}
\renewcommand\@dotsep{4.5}
\renewcommand\tableofcontents{%
    \section*{Contents}\ \vskip 5pt%
    \@starttoc{toc}%
    \AtEndDocument{\addtocontents{toc}{\endgroup}}%
    }
\renewcommand*\l@part[2]{%
  \ifnum \c@tocdepth >-2\relax
    \addpenalty\@secpenalty
    \addvspace{2.25em \@plus\p@}%
    \begingroup
      \setlength\@tempdima{3em}%
      \parindent \z@ \rightskip \@pnumwidth
      \parfillskip -\@pnumwidth
      {\leavevmode
       \large \bfseries #1\hfil \hb@xt@\@pnumwidth{\hss #2}}\par
       \nobreak
       \if@compatibility
         \global\@nobreaktrue
         \everypar{\global\@nobreakfalse\everypar{}}%
      \fi
    \endgroup
  \fi}
\renewcommand*\l@section{\@dottedtocline{1}{0pt}{1.8em}}
\renewcommand*\l@subsection{\@dottedtocline{2}{1.5em}{2.3em}}
\renewcommand*\l@subsubsection{\@dottedtocline{3}{3.8em}{3.2em}}
\renewcommand*\l@paragraph{\@dottedtocline{4}{7.0em}{4.1em}}
\renewcommand*\l@subparagraph{\@dottedtocline{5}{10em}{5em}}
\makeatother
\usepackage[dvips,
top=120pt,
bottom=120pt, 
footskip=35pt,
left=110pt, right=110pt,headheight=0pt, 
includehead,includefoot,heightrounded, 
]{geometry} 
\numberwithin{equation}{section}

\newtheorem{theorem}[equation]{Theorem}
\newtheorem{proposition}[equation]{Proposition}
\newtheorem{prop-def}[equation]{Proposition with Definition}

\newtheorem{remark}[equation]{Remark}
\newtheorem{lemma}[equation]{Lemma}
\newtheorem{example}[equation]{Example}

\newtheorem{working-hypothesis}[equation]{Working Hypothesis}

\newcommand{\tp}[1]{{}^t\kern-1pt#1}

\newcommand{\cardinarity}{\mbox{\tt \#}\hskip 1pt}

\renewcommand{\vec}[1]{\ensuremath{\mathchoice
                    {\mbox{\boldmath$\displaystyle#1$}}
                    {\mbox{\boldmath$\textstyle#1$}}
                    {\mbox{\boldmath$\scriptstyle#1$}}
                    {\mbox{\boldmath$\scriptscriptstyle#1$}}}}%
\renewcommand\i{\vec{i}}

\newcommand{\strutd}{\vrule height 0pt depth 6pt width 0pt}
\newcounter{item}
\newenvironment{oitem}{%
\begin{list}{{\rm (\arabic{item})}}
{\usecounter{item}
 \setlength{\topsep}{0pt}
 \setlength{\leftmargin}{18pt}
 \setlength{\labelsep}{5pt}
 \setlength{\labelwidth}{30pt}
 \setlength{\rightmargin}{0pt}
 \setlength{\parsep}{0.0pt}
 \setlength{\itemsep}{0.0pt}
 \setlength{\itemindent}{0pt}}}
{\end{list}}
{\end{list}}

\usepackage{tikz}
\usetikzlibrary{arrows.meta,arrows,chains,matrix,positioning,scopes,calendar}
\makeatletter
\tikzset{join/.code=\tikzset{after node path={%
\ifx\tikzchainprevious\pgfutil@empty\else(\tikzchainprevious)%
edge[every join]#1(\tikzchaincurrent)\fi}}}
\makeatother
\tikzset{>=stealth',every on chain/.append style={join},every join/.style={->}}
\tikzstyle{labeled}=[execute at begin node=\(\scriptstyle,execute at end node=\)]
\title{Arithmetic over the Gaussian Number Field on\\ a Certain Family of Elliptic Curves with \\ Complex Multiplication}
\TitleHead{Arithmetic on a Certain Family of Elliptic Curves}          
\author{Yoshihiro \textsc{\^Onishi}\footnote{Faculty of Science and Technology, Meijo University, Nagoya
468-8502, Japan.\newline e-mail: \texttt{yonishi@meijo-u.ac.jp}}
          ~and Fumio \textsc{Sairaiji}\footnote{Faculty of Nursing, 
          Hiroshima International University, Hiroshima
737-0112, Japan. \endgraf e-mail: \texttt{sairaiji@hirokoku-u.ac.jp}}}
\AuthorHead{\^Onishi and Sairaiji}         
\support{
This work was supported by the RIMS
}  
\VolumeNo{x}           
\YearNo{201x}           
\PagesNo{000--000}      
\communication{Received May 6, 2020.
Revised ????? ??, 2020.}      
\newcommand{\C}{\mathbf{C}}
\newcommand{\R}{\mathbf{R}}

\newcommand{\Q}{\mathbf{Q}}
\newcommand{\Z}{\mathbf{Z}}

\newcommand{\LT}{\mathbf{L\!T}}
\newcommand{\llangle}{\langle\!\langle}
\newcommand{\rrangle}{\rangle\!\rangle}

\newcommand{\ord}{\mathrm{ord}}
\renewcommand{\sl}{\widehat{\mathbf{sl}}}
\begin{document}
\setlength{\abovedisplayskip}{4.5pt} 
\setlength{\belowdisplayskip}{4.5pt} 
%

\maketitle

\begin{abstract}      
This work is a sequel of a previous work of one of the authors (Y.\^O), 
which treated certain congruence relation between an elliptic Gauss sum and 
a coefficient of power series expansion at the origin of the lemniscate sine function. 
We extend the previous result (in \cite{O}) which concern{\color{black}ed} only for non-vanishing elliptic Gauss sums. 
We give new congruence relations between power series coefficients of 
the lemniscate cosine function, which hold if and only if the corresponding 
elliptic Gauss sum vanishes. 
\end{abstract}
\allowdisplaybreaks
\noindent
\textbf{\large Introduction}\ 
\vskip 5pt
\noindent
In the paper \cite{H}, Hurwitz gave the following result\,:
\begin{theorem}\label{Hurwitz_congr} \ 
Let \,\(p>3\) \, be a rational prime, 
and let \,\(h(-p)\)\, be the class number of the imaginary quadratic field \,\(\mathbf{Q}(\sqrt{-p\,}\,)\).
Then we have
        \begin{equation*}
          h(-p)\equiv
          \Bigg\{
          \begin{aligned}
            \ \     -2\,B_{\frac{p+1}2}    \,\bmod{p} & \ \ \mbox{if \,\(p\equiv 3\bmod{4}\)},\\
            \        2^{-1}\,E_{\frac{p-1}2}\,\bmod{p} & \ \ \mbox{if \,\(p\equiv 1\bmod{4}\)}. 
          \end{aligned}
        \end{equation*}
Here \ \(B_n\) \ is the \(n\)-th Bernoulli number, 
and \,\(E_n\) \,is the \(n\)-th Euler number\footnote{\,We define \(E_n\) by \(\mathrm{sech}(u)=\sum_{n=1}^{\infty}(E_n/n!)u^n\). 
So that, \(E_2=-1\), \(E_4=5\), \(E_6=-61\), \(\cdots\). }.
Moreover, the absolutely smallest residue of the right hand side exactly equals to \,\(h(-p)\). 
\end{theorem}

Each of these congruences is proved by expressing 
the value at \(s=1\) of the Dirichlet \(L\)-series \ \(L(s,\Big(\tfrac{\cdot}{p}\Big))\) 
as a trigonometric Gauss sums, which is defined by a sort of Gauss sum using 
suitable trigonometric function instead of the exponential function in the classical Gauss sum. 
Under the Birch Swinnerton-Dyer (BSD) conjecture, 
one of the authors gave in \cite{O} an analogue of Theorem \ref{Hurwitz_congr} by 
replacing Dirichlet \(L\)-series and the trigonometric Gauss sum 
by Hecke's \(L\)-series and an ellipitc Gauss sum, respectively, 
in which the class number is replaced by a square root of the conjectural order of the Tate-Shafarevich group, 
and the Bernoulli number or Euler number is done 
by certain coefficient of the power series expansion at the origin of 
the lemniscate sine function. 

Elliptic Gauss sums were used, in order to compute numerically the \(L\)-series attached to some 
elliptic curves over \,\(\mathbf{Q}\), in the famous original paper \cite{BSD} by Birch and Swinnerton-Dyer. 
We wish to use them for investigation of \,\(L\)-series attached to some elliptic curves defined over \,\(\mathbf{Q}(\i)\). 

The paper \cite{O} is written about such investigation only for 
the case where the associated prime \(\ell\) is congruent to \(5\) or \(13\) modulo \(16\), 
since the treated elliptic Gauss sum for that case does not vanish and 
it is directly relates the order of Tate-Shafarevich group. 
In this paper, we extend the result \(\cite{O}\) to 
all the cases on modulo \(16\) of the primes  \(\ell\)  congruent \(1\) modulo \(4\). 
The remarkable point is that, in the cases which do not treated in \cite{O}, 
the corresponding elliptic Gauss sums indeed vanish often, 
which means the associated Hecke \(L\)-series vanish as well. 
We verify such vanishing phenomenon by the tables in \cite{A}. 
So that, the corresponding elliptic curve, which is defined over the Gaussian number field, 
is expected to be of positive Mordell-Weil rank. 

We present certain Kummer type congruences (Theorem \ref{0212a}) on power series coefficients of 
the lemniscate cosine function 
which are valid if and only if the corresponding elliptic Gauss sum 
(hence the value at \(1\) of the corresponding Hecke \(L\) series) vanishes. 

The corresponding elliptic curve 
(see the defining equations (\ref{model_13}), (\ref{model_5}), and (\ref{ell_curve_1mod8}))
is additive reduction modulo \(\lambda\) and our Kummer-type congruence is quite resemble to 
the Kummer congruence which guarantees the existence of the Kubota-Leopoldt \(p\)-adic \(L\)-function. 
So the authors hope that our result gives a hint for a construction of \(p\)-adic \(L\)-functions 
for an elliptic curve which is additive reduction modulo \(p\). 

This paper is organized as follows. 
From \S \ref{preliminaries} to \S \ref{ray_class_field}, we setup fundamental background. 
From \S \ref{On_Asai} to \S \ref{end_of_8n+5}, we review the results in \cite{A} and \cite{O}. 
In \S \ref{8n+1}, we review the result for the rest cases which is omitted in \cite{O}.  
In \S \ref{Gaussian_cong_numb}, we discuss some structure of the Mordell-Weil group of 
the curve and how our theory relates BSD conjecture. 
From \S \ref{main_result} to \S \ref{estimate}, we give the main result (Theorem \ref{0212a}) and its proof. 
Especially, in \S \ref{toward_p-adic_L}, 
we show a \textit{two term congruence relation} (see Theorem \ref{two_term_congr}) which might be 
a hint to construct a \(p\)-adic (\(\lambda\)-adic) \(L\)-function 
for an ellipitc curve that is additive reduction modulo \(\lambda\). 


\textit{Acknowledgment}: 
The authors thank Prof. G.~Yamashita who informed them of Hurwitz' paper \cite{H}, 
which looks the earliest literature in which Theorem \ref{Hurwitz_congr} apperared.
They also thank Prof. S.~Yasuda to whose advice we owe \S \ref{central_value}. 
\newpage
{\small
\tableofcontents
}
\vskip 15pt
\section{The lemniscate sine and cosine function\label{preliminaries}}\ 
\vskip 5pt
\noindent
The inverse function \ \(u\mapsto t\) \ of
\begin{equation*}
  t\mapsto u=\int_0^{\,t}\frac{dt}{\sqrt{1-t^{\,4}\,}\,}=\sum_{n=0}^{\infty}(-1)^n\binom{-\frac12}{n}\frac{t^{\,4n+1}}{4n+1}=t+\cdots
\end{equation*}
is called the \textit{lemniscate sine function}, which is denoted by \,\(t=\mathrm{sl}(u)\) 
and is expanded as
\begin{equation*}
    \mathrm{sl}(u)
   =u-\frac{1}{10}u^5+\frac{1}{120}u^9-\frac{11}{15600}u^{13}+\cdots
   =\sum_{n=0}^{\infty}C_n\,u^n
\end{equation*}
with \,\(C_n\)\, in \,\(\mathbf{Q}\). 
Then we have \(C_n=0\) if \(n{\not\equiv}1\bmod{4}\) and \(n!\,C_n\)\, belongs to \,\(\mathbf{Z}\). 
It is an elliptic function whose period lattice is \,\(\varOmega=(1-\i)\,\varpi\,\mathbf{Z}[\i]\), 
where
  \begin{equation}\label{expansion_sl}
  \varpi
  =2\int_0^{\,1}\frac{dt}{\sqrt{1-t^{\,4}\,}\,}
  =\int_1^{\,\infty}\frac{dx}{2\sqrt{x^3-x\,}\,}
  =2.62205\cdots. 
  \end{equation}
The divisor of \,\(\mathrm{sl}(u)\) modulo \,\(\varOmega\) \,is given by
 \begin{equation*}
 \mathrm{div}(\mathrm{sl})=(0)+(\varpi)-\left(\frac{\varpi}{1+\i}\right)-\left(\frac{\i\varpi}{1+\i}\right).
\end{equation*}
\vskip 5pt
\noindent
Throughout this paper, we denote 
\begin{equation*}
\varphi(u)=\mathrm{sl}\left(\,(1-\i)\,\varpi\,u\,\right).
\end{equation*}
The lemniscate cosine \,\(\mathrm{cl}(u)\)\, is defined by 
\begin{equation*}
\mathrm{cl}(u)=\mathrm{sl}\left(u+\frac{\varpi}2\right). 
\end{equation*}
Moreover, we use the notation
\begin{equation*}
\psi(u)=\mathrm{cl}\left((1-\i)\thinspace\varpi u\right). 
\end{equation*}
Then both of functions \(\varphi(u)\) and \(\psi(u)\) have
period lattice \,\(\mathbf{Z}[\i]\)\, of \,\(\mathbf{C}\). 
We define the \,\(D_n\)'s \,by the expansion of \,\(\mathrm{cl}(u)\)\, as
\begin{equation}\label{coeff_D}
  \begin{aligned}
    \mathrm{cl}(u)
   =\sum_{n=0}^{\infty}D_n\,u^n
  &=1-u^2+\frac12u^4-\frac{3}{10}u^6+\frac{7}{40}u^8-\cdots. 
  \end{aligned}
\end{equation}
Then, \(D_n=0\)\, for odd \,\(n\)\, and \,\(n!\,D_n\)\, is in \,\(\mathbf{Z}\). 
\vskip 15pt
\section{The ray class field\label{ray_class_field}}\ 
\vskip 5pt
\noindent
We take a rational prime \ \(\ell\equiv 1\bmod{4}\), 
and we fix its decomposition 
\(\ell=\lambda\overline{\lambda}\)\, in \,\(\mathbf{Z}[\i]\)\, with \,\(\lambda\equiv 1\bmod{(1+\i)^3}\). 
We fix a subset \,\(S\)\, of \,\(\mathbf{Z}[\i]\)\, (sometimes called a \(1/4\)-set) such that 
\,\(\left(\mathbf{Z}[\i]/({\lambda})\right)^{\times}\simeq S\cup -S\cup \i S\cup -\i S\), 
\(|S|=(\ell-1)/4\).
Moreover we define
\begin{equation}\label{basic_materials}
\begin{aligned}
    \varLambda&=\varphi\big(\tfrac1{\lambda}\big), \ \ 
    \mathscr{O}_{\lambda}=\mbox{\lq\lq\,the ring of integers in \ \(\mathbf{Q}(\i,\varLambda)\)\,''}, \\
  \tilde{\lambda}&=\gamma(S)^{-1}\,\prod_{r\,\in\,S}\varphi\big(\tfrac{r}{\lambda}\big),\\
&\mbox{where}\ \ \ \ 
  \left\{
  \begin{aligned}
  \{\pm1,\ \pm\i\}     \ni\gamma(S)\ \equiv\prod_{r\,\in\,S}r\  \bmod{\lambda} & \ \ \mbox{if \,\(\ell\equiv 5 \bmod{8}\)}, \\
  \hskip 20pt \{\pm\i\}\ni\gamma(S)^2\equiv\prod_{r\,\in\,S}r^2 \bmod{\lambda} & \ \ \mbox{if \,\(\ell\equiv 1 \bmod{8}\)}. 
  \end{aligned}
  \right.
\end{aligned}
\end{equation}
Here, we have \(\pm\) sign ambiguity of \(\gamma(S)\) in the case of \,\(\ell\equiv 1 \bmod{8}\). 
In any case, we know (see for example {\color{black}p.106 of \cite{A} or} Lemma 1.11 \cite{O}) that
\begin{equation}\label{varLambda_lambda}
\varLambda\in\mathscr{O}_{\lambda}, \ \ \ 
(\varLambda)^{\ell-1}=(\lambda), \ \ \  
\tilde{\lambda}^4={\color{black}-}\lambda.
\end{equation}
Note that \,\(\mathbf{Q}(\i,\varLambda)\)\, is the ray class field over \,\(\mathbf{Q}(\i)\)\, 
of conductor \ \(((1+\i)^3\lambda)\) (see Takagi \cite{T}, \S 32, for instance). 

Throughout this paper, we fix the identification
\begin{equation}\label{Z_ell_Zi_lambda}
\mathbf{Z}[\i]_{\lambda}\simeq\mathbf{Z}_{\ell},
\end{equation}
where the left hand side is the \(\lambda\)-adic completion of \,\(\mathbf{Z}[\i]\). 
Moreover, we consider they are subring of the algebraic closure \,\(\overline{\Q_{\ell}}\)\, of \,\(\Q_{\ell}\). 
We use the following notation. 
For any element \,\(\alpha\)\, in the integer ring \,\(\overline{\Z_{\ell}}\)\, 
of \,\(\overline{\Q_{\ell}}\), we denote the \(\ell\)-adic order by \,\(\ord\). 
For example, for \,\(\alpha\)\, in \,\(\Z[\i]_{\lambda}\), 
\,\(\ord(\alpha)=n\)\, if and only if 
\,\(\lambda^n\)\, divides \,\(\alpha\)\, but \,\(\lambda^{n+1}\)\, does not. 
\vskip 15pt
\section{Asai's theory\label{On_Asai}}\ 
\vskip 5pt
\noindent
In this section we recall the results from \cite{A} in order to go to the rest cases smoothly. 
We assume here \,\(\ell\equiv 13\bmod{16}\)\, for simplicity. 
For the other cases, see \cite{A}. 
We put \,\(\chi_{{}_\lambda}(r)=\smash{\left(\dfrac{r}{\lambda}\right)_{\!4}}\). 
In this case, the \textit{elliptic Gauss sum} is defined by
\begin{equation*}
\mathrm{egs}(\lambda)=\sum_{r\in S}\chi_{{}_\lambda}(r)\,\varphi\Big(\frac{r}{\lambda}\Big).
\end{equation*}%
Since the terms of this summation are algebraic integers, 
so is \,\(\mathrm{egs}(\lambda)\). 
\begin{theorem}\label{egs_Asai}
{\rm (\cite{A})}\ 
There exists an odd \,\(A_{\lambda}\) in \,\(\mathbf{Z}\)\, such that
\begin{equation*}
\mathrm{egs}(\lambda)=A_{\lambda}\,\tilde{\lambda}^3, 
\end{equation*}
where 
\(\tilde{\lambda}\) is defined by {\rm(\ref{basic_materials})}. 
In particular, \,\(\mathrm{egs}(\lambda)\neq0\).
\end{theorem}
\begin{remark}{\rm
(1) This theorem is proved by using 
the functional equation for the Hecke \(L\)-series corresponding to 
the suitable Hecke character associated to \,\(\chi_{{}_\lambda}\)\, 
and the formula of Cassels-Matthews (see \cite{M}) for the classical quartic Gauss sum which appears as 
the root number of the functional equation. 
It is expected to have another prove the formula of Cassels-Matthews 
if we get a part of BSD conjecture including the parity of the order of the corresponding Tate-Shafarevich group. \\
(2) We call \,\(A_{\lambda}\)\, the \textit{coefficient} of \,\(\mathrm{egs}(\lambda)\)\, according to \cite{A}. 
}
\end{remark}
We recall the corresponding Hecke \(L\)-series. 
Still we are assuming \,\(\ell\equiv 13\bmod{16}\). 
Taking \,\(\{\,1,\,\i\,\}\)\, as a set of complete representatives of \,\(\left(\mathbf{Z}[\i]\,\big/\,(1+\i\,)^2\right)^{\times}\), 
we define
\begin{equation*}
{\chi_{{}_0}}(\alpha)=\varepsilon^2\ \ 
    \mbox{for} \ \ \alpha\equiv\varepsilon\bmod{(1+\i)^2}, \ \varepsilon\in\{\,1,\ \i\,\}, \ \ \ \ 
\widetilde{\chi}((\alpha))=\chi_{{}_\lambda}(\alpha)\,{\chi_{{}_0}}(\alpha)\,\overline{\alpha}.
\end{equation*}%
Then \,\(\widetilde{\chi}\)\, is a Hecke character of conductor \ \(\left((1+\i)^2\lambda\right)\). 
Now we have the following expression given by Asai \cite{A} for the central value of the corresponding Hecke \(L\)-series. 
\begin{theorem}\label{Hecke_L_val}
We have \ 
\(L(1,\,\widetilde{\chi})=-\varpi\,(1-\i)^{-1}\chi_{{}_\lambda}(2)\lambda^{-1}\,\mathrm{egs}(\lambda)\).          
\end{theorem}
Searching an elliptic curve whose conductor is the square of that of \(\widetilde{\chi}\) 
(see \cite{ST}, Theorem 12), 
we see that the elliptic curve corresponding to \,\(L(s,\widetilde{\chi})\)\, is 
\begin{equation}\label{model_13}
\mathscr{E}_{-\lambda}\,:\, y^2=x^3+\lambda x, \ \ \ (\,\lambda\overline{\lambda}=\ell\equiv 13 \bmod{16}\,).
\end{equation}
Deuring \cite{D} showed that%
\begin{equation}\label{Deuring}
L_{\mathscr{E}_{-\lambda}/\mathbf{Q}(\i)}(s)=L(s,\,\widetilde{\chi})\,L(s,\,\overline{\widetilde{\chi}}\,).
\end{equation}
Especially, if \,\(\ell\equiv 13\bmod{16}\), then 
\,\(\mathrm{rank}\,\mathscr{E}_{-\lambda}\left(\mathbf{Q}(\i)\right)=0\),
which is shown by Theorems \ref{egs_Asai}, \ref{Hecke_L_val}, 
and Coates-Wiles theorem in \cite{CW}. 
Moreover, we recall the following result from \cite{O}. 
\begin{proposition}\label{TS_and_A}
If the full statement of BSD conjecture for the curve \,\(\mathscr{E}_{-\lambda}\)\, 
is true, then \,
\(\cardinarity\,\Sha\big(\mathscr{E}_{-\lambda}/\mathbf{Q}(\i)\big)={A_{\lambda}}^2\). 
\end{proposition}
\vskip 5pt
\section{Some congruence on the coefficients of elliptic Gauss sums}\label{end_of_8n+5}\ 
\vskip 5pt
\noindent
The former part of the following theorem is proved in \cite{O}
and reproved a sophisticated method as Lemma \ref{0214d} later. 
Let \,\(C_n\)\, be the coefficient of \,\(u^j\)\, defined by {\rm(\ref{expansion_sl})}. 
Since \,\(\frac34(\ell-1)!\,C_{\frac34(\ell-1)}\)\, is in \,\(\mathbf{Z}\), 
\,\(-\frac14\,C_{\frac34(\ell-1)}\)\, is in \,\(\mathbf{Z}_{\ell}\). 
\begin{theorem}\label{O_congr13}
\!\!\!{\rm (\cite{O})} \ 
Assuming \ \(\ell\equiv 13\ \bmod{16}\), \,we have
\,\(A_{\lambda}\equiv-\tfrac14\,C_{\frac34(\ell-1)}\,\bmod{\ell}\). 
The absolutely minimal residue of the right hand side is exactly equal to \,\(A_{\lambda}\). 
\end{theorem}%
\noindent
The latter part of Theorem \ref{O_congr13} follows from the former part and 
the following lemma which is proved in \S \ref{estimate}. 
\begin{lemma}\label{less_than_l_0}
For any \(\ell=\lambda\overline{\lambda}\equiv1\mod{4}\), we have 
\,\(|A_{\lambda}|<\tfrac12\,\ell\). 
\end{lemma}
\begin{remark}{\rm
Observing Kanou's manmouth table, 
the behavior of \,\(|\,\mathrm{egs}({\lambda})\,|\)\, with respect to \,\(\ell\to\infty\)\, 
is quite small. 
Indeed, the estimation  \,\(|\,A_{\lambda}\,|<\ell^{1/4}\)\, is hopeful. 
}
\end{remark}
Joining Proposition \ref{TS_and_A} and Therem \ref{O_congr13} together, 
we have a natural generalization of Hurwitz' congruence in Theorem \ref{Hurwitz_congr}. 
\par
For the case of \,\(\ell\equiv 5\bmod{16}\), we have a similar story which is described in \cite{A} and \cite{O}. 
The corresponding ellipitic curve for this case is 
\begin{equation}\label{model_5}
\mathscr{E}_{\frac14\lambda}\,:\,y^2=x^3-\tfrac14\lambda x, \ \ \ (\,\lambda\overline{\lambda}=\ell\equiv 5 \bmod{16}\,),
\end{equation}
for which we have \,\(\mathrm{rank}\ \mathscr{E}_{\frac14\lambda}\left(\mathbf{Q}(\i)\right)=0\), 
and the corresponding congruence as (\ref{O_congr13}). 

So, from the next section, we proceed to the remaining case of \,\(\ell\equiv 1\bmod{\,8}\). 
About \,\(18\)\%  of the \,\(172\)\, examples of this case in \,\cite{A}, 
\(\mathrm{egs}(\lambda)=0\)\, 
holds, where the ellipitic Gauss sum for this case is defined in the next section. 
\vskip 10pt
\section{The foregoing researches in the case of \,\(\ell\equiv1\) mod \(8\)\label{8n+1}}\ 
\vskip 5pt
\noindent
From now on, we always assume the prime \,\(\ell\)\, satisfies \,\(\ell\equiv1\bmod{8}\),
\(\ell={\lambda}\overline{{\lambda}}\), 
\({\lambda}\equiv 1\bmod{(1+\i)^3}\), 
\(\chi_{{}_\lambda}(\nu)=\left(\dfrac{\nu}{\lambda}\right)_{\!4}\). 
Then we see \,\(\chi_{{}_\lambda}(\i)=\i^{\frac{\ell-1}4}=(-1)^{\frac{\ell-1}{8}}\). 
Using  \,\(\psi(u)=\mathrm{cl}\left((1-\i)\thinspace\varpi u\right)\), 
the elliptic Gauss sum for this case is defined by 
   \begin{equation*}
   \mathrm{egs}({\lambda})=\sum_{\nu\in S\cup\i S}\chi_{{}_\lambda}(\nu)\,\psi\bigg(\frac{\nu}{{\lambda}}\bigg). 
   \end{equation*}
In this paper \,\(\varepsilon\)\, always denotes any element in  \,\(\vec{\mu}_4\), 
where \,\(\vec{\mu}_4=\{1,\,-1,\,\i,\,-\i\}\). 
Recalling 
the canonical isomorphism \,\(\vec{\mu}_4\overset{\sim}{\rightarrow}\big(\mathbf{Z}[\i]/(1+\i)^3\big)^{\times}\), 
we define the character \,\(\chi_{{}_0}\)\, by
\begin{equation*}
\chi_{{}_0}(\alpha)=\varepsilon \ \ \ \mbox{if}\ \ \ \alpha\equiv\varepsilon\bmod{(1+\i)^3} \ \ \ \ \ \ 
(\ \alpha\in\mathbf{Z}[\i],\   (1+\i){\not|}\,\alpha\ ). 
\end{equation*}
(Case 1) \ If \,\(\ell\equiv1\bmod{16}\), \(\chi_{{}_\lambda}(\i)=1\). 
We define \,\(\chi_{{}_1}=\chi_{{}_\lambda}\chi_{{}_0}\)\, 
and \,\(\widetilde{\chi}((\alpha))=\chi_{{}_1}(\alpha)\,\overline{\alpha}\). \\
(Case 2) \ If \,\(\ell\equiv9\bmod{16}\), \(\chi_{{}_\lambda}(\i)=-1\). 
So defining \,\(\chi_{{}_1}=\chi_{{}_\lambda}\overline{\chi_{{}_0}}\)\, 
and \,\(\widetilde{\chi}((\alpha))=\chi_{{}_1}(\alpha)\,\overline{\alpha}\).
\par
In any case, we see \,\(\widetilde{\chi}\)\, is a Hecke character of conductor \ \(((1+\i)^3\lambda)\). 
Then, as in \cite{A}, we have the following expression :
\begin{equation}\label{Hecke_L_egs}
L(1,\widetilde{\chi})=(-1)^{\frac18(\ell-1)}\,\varpi\,
\overline{\chi_{{}_\lambda}(1+\i)}\,2^{-1}\lambda^{-1}\,\mathrm{egs}({\lambda}).     
\end{equation}
\begin{theorem}\label{egs_A_lambda}
{\rm (\cite{A})} Let \,\(\zeta_8=\exp(2\pi\i/8)\). 
There exists \,\(A_{\lambda}\)\, in \,\(\mathbf{Z}[\zeta_8]\)\, such that 
  \begin{equation}\label{model_1}
  \mathrm{egs}(\lambda)=A_{\lambda}\,{\tilde{\lambda}\,}^3.
  \end{equation}
Here, \,\(A_{\lambda}\)\, is given by the table \,{\rm(\ref{table1})}\, below 
with some \,\(a_{\lambda}\)\, in \,\(\mathbf{Z}\). 
\end{theorem}

This theorem is also proved by using the formula of Cassels-Matthew and 
the functional equation of \ \(L(s,\widetilde{\chi})\). 
In \cite{A}, it is observed by Asai that \,\(a_{\lambda}\)\, is in \,\(2\mathbf{Z}\), 
but any proof of this is not known yet. 
\par
Searching the elliptic curve whose conductor is  \,\(\big(\,(1+\i)^3\lambda\,\big)^2\), 
which is square of that of \(\widetilde{\chi}\)\, (\cite{ST}, Theorem 12), 
we see the Hecke \(L\)-series associated to \,\(\mathrm{egs}(\lambda)\)\, 
is a factor of the \(L\)-series of the elliptic curve
\begin{equation}\label{ell_curve_1mod8}
\mathscr{E}_{\lambda}\,:\,  y^2=x^3-\lambda x,  
 \ \ \ (\,\lambda\overline{\lambda}=\ell\equiv 1 \bmod{8}\,).
\end{equation}
We have the same equation as (\ref{Deuring}) for this case as well.
The reduction type at \,\((1+\i)\)\, is of type \,\(\mathrm{III}\), 
and one at \,\(\lambda\)\, is of type \,\({\mathrm{I}_2}^*\). 
Each Tamagawa number \,\(\tau_{\mathfrak{p}}\)\, 
and the coefficients \,\(A_{\lambda}\)\, of \,\(\mathrm{egs}(\lambda)\)\, 
is given as follows\,:
\begin{equation}\label{table1}
\arraycolsep=7pt
\def\arraystretch{1.1}
\begin{array}{c|c|cccc}
\multicolumn{2}{c|}{\chi_{{}_\lambda}(1+\i)}    & \strutd        1                   &                  -1            &             i                      &            -i                  \\\hline
                     & A_{\lambda}             & \i\!\sqrt{2}\,{\cdot}\,a_{\lambda} & \sqrt{2}\,{\cdot}\,a_{\lambda} & \zeta_8\,{\cdot}\,a_{\lambda}      & \i\zeta_8\,{\cdot}\,a_{\lambda} \\[-2pt]
\ell\equiv1\bmod{16} & \tau_{(\lambda)}        &                2                   &                   2            &             2                      &             2                   \\[-3pt]
                     & \tau_{(1+\i)}\strutd &                4                   &                   4            &             2                      &             2                      \\\hline
                     & A_{\lambda}             & \i\zeta_8\,{\cdot}\,a_{\lambda}    & \zeta_8\,{\cdot}\,a_{\lambda}  & \i\!\sqrt{2}\,{\cdot}\,a_{\lambda} &  \sqrt{2}\,{\cdot}\,a_{\lambda} \\[-2pt]         
\ell\equiv9\bmod{16} & \tau_{(\lambda)}        &                2                   &                   2            &             2                      &             2                   \\[-3pt]
                     &    \tau_{(1+\i)}        &                2                   &                   2            &             4                      &             4                   
\end{array}
\end{equation}
\begin{remark}{\rm
Assuming the full statement of BSD conjecture true, 
we have \,\({\left(\,\frac12\,a_{\lambda}\,\right)}^2=\cardinarity\,\Sha(\mathscr{E}_{\lambda})\)\,
if \,\(a_{\lambda}\neq0\). 
}
\end{remark}
\vskip 5pt
\noindent
Recall the numbers \,\(D_n\)\, defined in (\ref{coeff_D}). 
Since \,\((\frac{3}{4}(\ell-1))!\,D_{\frac{3}{4}(\ell-1)}\)\, is in \,\(\mathbf{Z}\), 
\(-\frac{1}{2}D_{\frac{3}{4}(\ell-1)}\)\, is in \,\(\mathbf{Z}_{\ell}\). 
We keep in mind that \,\(\mathbf{Z}[\zeta_8]\)\, is an Euclidean ring. 
Using the method of \cite{O} and Lemma \ref{less_than_l_0}, the following is shown. 
\begin{theorem}
\label{cong_for_1_mod_8}
Let \,\(\widetilde{\lambda}_0\)\, be a prime lying above \,\(\lambda\)\, in \,\(\mathbf{Q}(\zeta_8)\). 
We have 
\begin{equation*}
A_{\lambda}\equiv -\frac{1}{2}\,D_{\frac{3}{4}(\ell-1)}\bmod{\widetilde{\lambda}_0},
\end{equation*}
where \,\(A_{\lambda}\)\, is given by the table \,{\rm(\ref{table1})}. 
Furthermore, \(A_{\lambda}\)\, is the minimal residue in \,\(\zeta_8\,\mathbf{Z}[\i]\)\, of the right hand side 
with respect to the absolute norm. 
\end{theorem}
\vskip 6pt
\section{An analogue of the congruent numbers\label{Gaussian_cong_numb}}\ 
\vskip 5pt
\noindent
The following is well-known (see, for example, Koblitz' book \cite{K}). 

\begin{proposition}
\label{Koblitz}
Let \,\(n\)\, be a rational integer. 
For the elliptic curve {\rm\,\(\mathscr{E}_{n^2}\)\,: \(y^2=x^3-n^2x\)}, 
the following three are equivalent each other\,{\rm:}
\begin{oitem}
\item There exist \(u\), \(v\) in \,\(\mathbf{Q}\)\, such that \,\(n^2=u^4-v^2\)\,{\rm;}
\item \(n\)\, is a congruent number\,{\rm;}
\item \(\mathrm{rank}\ \mathscr{E}_{n^2}(\mathbf{Q})>0\).
\end{oitem}
\end{proposition}
{\color{black}
The claim (1) is a sort of paraphrase of the definition of congruent number for \(n\). }
The equivalence of (2) and (3) is described as Proposition 18 in \cite{K}. 

\begin{lemma}\label{nagell} 
Let \,\(A\)\, be a square-free integer in \,\(\mathbf{Z}[\i]\)\, not dividing \(5\). 
Then there are only two torsion points \,\((0,0)\)\, and \,\(\infty\) in the group of 
\,\(\mathbf{Q}(\i)\)-rational points on the elliptic curve
\vskip -15pt
\begin{equation*}
\mathscr{E}_A\,:\,y^2=x^3-Ax.
\end{equation*}
\end{lemma}
\vskip 3pt
\begin{proof}
This proof is a slight modification of the argument in \cite{Nagell}. 
Since \,\(A\)\, is square-free, the equation \,\(x^3-Ax=0\)\, has only root \,\(x=0\)\, in \,\(\mathbf{Q}(\i)\). 
Thus the \,\(\mathbf{Q}(\i)\)-rational point of \,\(\mathscr{E}_A\)\, of order two is \,\((0,0)\). 
Let \,\((a,b)\)\, be a \,\(\mathbf{Q}(\i)\)-rational point of \,\(\mathscr{E}_A\).  
The \,\(x\)-coordinate of \,\([1+\i](a,b)\)\, is \,\(x_{1+\i}=\big(\frac{b}{(1-\i)a}\big)^2\)\, 
and the \,\(x\)-coordinate of \,\([2](a,b)\)\, is \,\(x_{2}=\big(\frac{a^2-A}{2b}\big)^2\). 
Therefore, \,\(x\)-coodinate of any point in either \,\([1+\i]\mathscr{E}_A(\mathbf{Q}(\i))\)\, 
or \,\([2]\mathscr{E}_A(\mathbf{Q}(\i))\)\, is square. 
Assume \,\((a,b)\)\, is of finite order. 
Then \,\(a\)\, and \,\(b\) belong to \,\(\mathbf{Z}[\i]\)\,
(see \cite{Nagell}, p.14, Theorem 2 or \cite{C}, \S 11 and \S 12).  
If \,\((a,b)\)\, satisfies \,\([2](a,b)=(0,0)\), we have \,\(a^2=-A\). 
It does not occur because \,\(A\)\, is square-free. 
Hence there does not exitst any \,\(\mathbf{Q}(\i)\)-rational point of order divided by \(4\). 
Assume that \,\((a,b)\)\, is \,\(\mathbf{Q}(\i)\)-rational points of odd order. 
Since \,\(\mathscr{E}_A(\mathbf{Q}(\i))/[1+\i]\mathscr{E}_A(\mathbf{Q}(\i))\)\, is an abelian group of exponent two, 
\((a,b)\)\, is in \,\([1+\i]\mathscr{E}_A(\mathbf{Q}(\i))\). 
Thus \,\(a\)\, is square in \,\(\mathbf{Z}[\i]\). 
Since \,\([1+\i](a,b)\)\, is of odd order and \,\(x_{1+\i}\)\, is in \,\(\mathbf{Z}[\i]\), 
we have \,\(a\,|\,b\)\, and \,\(1+\i\,|\,b\).
As \,\(a\)\, is square and \,\(b^2=a(a^2-A)\), 
we have \,\(a=f^2\), \(b=f^2g\), \(a^2-A=f^2g^2\)\, for some \,\(f\), \(g\)\, in \,\(\mathbf{Z}[\i]\). 
Since \,\(-A=f^2(g^2-f^2)\)\, 
and \,\(A\)\, is square-free, \(f^2\)\, is unit. 
Thus we have \,\(f^2=\pm 1\).
Furthermore, \,\([2](a,b)\)\, is of odd order and \,\(x_{2}\)\, is in \,\(\mathbf{Z}[\i]\), 
we have \,\(2b\,|\,a^2+A\). 
Since \,\(f^2\)\, is unit, we have \,\(2g\,|\,2f^2-g^2\). 
Thus we have \,\(1+\i\,|\,g\)\, and \,\(\frac{g}{1+\i}\,\big|\,f^2\).
Since \,\(f^2\)\, is an unit, \,\(g\)\, is equal to \,\(1+\i\)\, up to unit. 
Therefore we have \,\(A=\pm (-1\pm 2i)\). 
They does not occur because \,\(A\)\, does not divide \(5\). 
Hence there does not exitsts any \,\(\mathbf{Q}(\i)\)-rational point of odd order. 
This completes the proof. 
\end{proof} 

\begin{remark}{\rm
In the last four cases of \,\(A\)\, 
in the proof above, we see the groups of \(\mathbf{Q}(\i)\)-rational points of the curves are of rank \(0\) 
because the \(L\)-functions do not vanish at \(1\) (see the proof of Lemma 2.11 (b) p.105, \cite{A}). 
So that they are finite groups due to \cite{CW}. \texttt{MAGMA} says that the groups are of order \(10\) generated by 
\(\pm(1\mp 2\i,\,-1\mp 3\i)\).
}
\end{remark}

We prove the following analogue of Proposition \ref{Koblitz}. 

\begin{proposition}\label{GCN}
Let \(\lambda\) be any Gaussian prime of degree \(1\) satisfying 
\,\(\lambda\equiv 1\bmod{(1+\i)^3}\)\, and assume \,\(\lambda{\not|}\ 5\). 
The following three statements are equivalent\,{\rm:}
\begin{oitem}
\item 
There are infinitely many \,\(\mathbf{Q}(\i)\)-rational points on \,\(\mathscr{E}_{\lambda}\),
namely, 
\begin{equation*}
\mathrm{rank}\,\mathscr{E}_{{\lambda}}\left(\mathbf{Q}(\i)\right)>0;
\end{equation*}
\item 
The prime \,\({\lambda}\)\, is of the form \,\(-\alpha^4+\beta^2\i\)\, with 
\,\(\alpha\), \(\beta\in\mathbf{Q}(\i)\); 
\item 
The prime \,\({\lambda}\)\, is of the form \,\(u^4-v^2\)\, with 
\,\(u\), \(v\in\mathbf{Q}(\i)\). 
\end{oitem}
\end{proposition}
\begin{proof}
(2)\(\Rightarrow\)(1). 
For the given expression \,\(\lambda=-\alpha^4+\beta^2\i\), 
we see \,\((\alpha^2\i,\,\alpha\beta)\)\, 
is a \(\mathbf{Q}(\i)\)-rational point of infinite order on the curve \,\(\mathscr{E}_{\lambda}\)\, 
because of Lemma \ref{nagell} and 
\begin{equation*}
(-\alpha^2\i)^3-\lambda(\alpha\i)^2
=(-\alpha^2\i)^3-(-\alpha^4+\beta^2\i)(\alpha\i)^2
=(\alpha\beta)^2.
\end{equation*}
(3)\(\Rightarrow\)(1). 
This is proved similarly. 
Indeed, if \,\(\lambda=u^4-v^2\), 
then \((x,y)=(u^2,\,uv)\)\, is a point of infinite order on \,\(\mathscr{E}_{\lambda}(\mathbf{Q}(\i))\) 
because of \,\(x^3-{\lambda}\,x=u^6-(u^4-v^2)\,u^2=(uv)^2 = y^2\). 
(1)\(\Rightarrow\)(3). 
Lemma \ref{nagell} implies that  
the set of torsion points of \,\(\mathscr{E}_{\lambda}(\mathbf{Q}(\i))\)\, 
is \,\(\{(0,0),\infty\}\). 
So we assume there exists a non-torsion point \,\((a,b)\), 
namely \,\(b^2=a^3-\lambda a\), 
with \,\(a\), \(b\)\, in \,\(\mathbf{Q}(\i)\). 
The duplication \,\([2](a,b)\)\, is given by
\begin{equation*}
\Big(\,\frac{(a^2+\lambda)^2}{4b^2}, \ 
\frac{a^6-5\lambda a^4-5\lambda^2a^2+\lambda^3}{8b^3}\,\Big). 
\end{equation*}
We define
\begin{equation*}
\smash{
u=\frac{a^2+\lambda}{2b}\ (\neq 0), \ 
v=\frac{a^4-6\lambda a^2+\lambda^2}{4b^2}. }
\end{equation*}
Then the point \,\((u^2,\,uv)\)\, is on the curve, 
and \,\(\lambda=u^4-v^2\).
\\
(1)\(\Rightarrow\)(2). 
This proof is given by \(2\)-descent, 
which is a modification of the proof of Proposition 1.4 in  Chapter X, \cite{Si}.  
We put 
\begin{equation*}
T_{\lambda}
=\{\,b\in \mathbf{Q}(\i)^{\times}/(\mathbf{Q}(\i)^{\times})^2
\ \big|\ 
\mathrm{ord}_{\pi}(b)\equiv 0 \bmod{2}\ \ \mbox{for all prime}\ \,\pi{\not\vert}\,\lambda\,\}. 
\end{equation*}
This is a subgroup of \,\(\mathbf{Q}(\i)^{\times}/(\mathbf{Q}(\i)^{\times})^2\)\, 
of order four generated by \,\(i\)\, and \,\(\lambda\). 
There is a homomorphism 
\vskip -20pt
\begin{equation}\label{a}
\mathscr{E}_{\lambda}(\mathbf{Q}(\i))
\rightarrow 
T_{\lambda}\ \ \mbox{defined by}\ \ 
(x,y)\mapsto
\begin{cases}
\       x & \mbox{if}\ x\ne 0, \\[-3pt]
\ \lambda & \mbox{if}\ x=0,    \\[-3pt]
\       1 & \mbox{if}\ x=\infty.
\end{cases}
\end{equation}
Indeed, if we put \,\((x_3,y_3)=(x_1,y_1)+(x_2,y_2)\)\, 
which is an addition on \,\(\mathscr{E}_{\lambda}\), and 
\begin{equation*}
m=\Bigg\{\ 
\begin{aligned}
(y_1-y_2)/(x_1-x_2)      \ & \ \mbox{if} \ \ x_1\ne x_2,\\
(3x_1^2+{\lambda})/(2y_1)\ & \ \mbox{if} \ \ x_1=x_2, 
\end{aligned}
\end{equation*}
we have \,\(x_1 x_2 x_3=(-m x_1+y_1)^2\).
Hence, \(x_3\in x_1x_2\,\big(\mathbf{Q}(\i)^{\times}\big)^2\)\, 
if \,\(x_1x_2\neq0\). 
If \,\(x_2\ne0\)\, and \,\(x_1=0\), we have 
\begin{equation*}
x_3=\left(\frac{y_2}{x_2}\right)^2-x_2
=\frac{-{\lambda}x_2}{{x_2}^2}\in
{\lambda}x_2 \big(\mathbf{Q}(\i)^{\times}\big)^2
\end{equation*}
because of \,\(m^2=x_1+x_2+x_3\).
We show that the kernel of (\ref{a}) is \,\([1+\i]\mathscr{E}_{\lambda}(\mathbf{Q}(\i))\). 
Let \,\((x_1,y_1)\)\, be a point in \,\(\mathscr{E}_{\lambda}(\mathbf{Q}(\i))\), 
\((x_3,y_3)=[1+\i](x_1,y_1)\)\, and \((x_2,y_2)=[\i](x_1,y_1)=(-x_1,{\i}y_1)\). 
Then we see
\begin{equation*}
x_3=m^2-x_1-x_2=\left(\frac{y_1}{(1+\i)x_1}\right)^2\in (\mathbf{Q}(\i)^{\times})^2. 
\end{equation*}
Therefore \,\([1+\i]\mathscr{E}_{\lambda}(\mathbf{Q}(\i))\)\, is contained in the kernel. 
Conversely, suppose \,\((x_3,y_3)\)\, is in \,\(\mathscr{E}_{\lambda}(\mathbf{Q}(\i))\)\, 
and \,\(x_3\)\, is a square. 
Then, firstly, \,\(m=y_1/((1+\i)x_1)\)\, is in \,\(\mathbf{Q}(\i)\) because of 
the assumption on \,\(x_3\). 
Secondly, as
\begin{equation*}
y_3=-m(x_3-x_1)-y_1=-m^3-\frac{1+\i}{2}y_1
\end{equation*}
is in \,\(\mathbf{Q}(\i)\), \,\(y_1\)\, is in \,\(\mathbf{Q}(\i)\). 
Thirdly, as \,\(m\)\, is in \,\(\mathbf{Q}(\i)\), 
\(x_1\)\, is also in \,\(\mathbf{Q}(\i)\). 
Accordingly, the kernel is contained in \,\([1+\i]\mathscr{E}_{\lambda}(\mathbf{Q}(\i))\). 
Therefore the induced homomorphism 
\begin{equation*}
\mathscr{E}_{\lambda}(\mathbf{Q}(\i))/[1+\i]\mathscr{E}_{\lambda}(\mathbf{Q}(\i))
\longrightarrow T_{\lambda}
\end{equation*}
is bijective. 
Summing up, for a \,\(\mathbf{Q}(\i)\)-rational point \,\((x,y)\), 
either \,\(x\)\, or the first coordinate of \,\((x,y)+(0,0)\)\, 
is of the form \,\(\alpha^2\i\). 
We write the obtained point as
\,\((\alpha^2\i,\,\alpha\beta)\). 
Then \,\(\alpha^2\beta^2=-\alpha^6\i-\lambda\alpha^2\i\). 
This means \,\(\lambda=-\alpha^4+\beta^2\i\)\, 
and the proof has been completed. 
\end{proof}
\begin{remark}{\rm We shall give some remarks on Proposition \ref{GCN}. 
\begin{oitem}
\item A prime \,\(\lambda\)\, of the form in (2) or (3) of Proposition \ref{GCN} should be called a \textit{Gaussian congruent number}. 
\item In the examples in \,\cite{A} \,each of statements (1), (2), and (3) of Proposition \ref{GCN} is satisfied 
if and only if \,\(\mathrm{egs}(\lambda)=0\). 
\item In the examples of \,\cite{A}\, such that \,\(\mathrm{egs}({\lambda})=0\), 
except \,\({\lambda}\overline{{\lambda}}=4817\equiv 1\bmod{16}\), 
we can take \,\(\alpha\), \(\beta\)\, in \,\(\mathbf{Z}[\i]\). 
See Example \ref{4817} below. 
\item We summarize the situation as follows:
\begin{equation*}
\begin{aligned}
\mbox{\({\lambda}\) \ is of the form \(-\alpha^4\i+\beta^2\)}\ 
&\xLongleftrightarrow{\hskip 18pt\mbox{\tiny Prop.\,\ref{GCN}}\hskip 20pt } \ 
\mathrm{rank}\,\mathscr{E}_{{\lambda}}\left(\mathbf{Q}(\i)\right)>0 \\
&\ \textcolor{gray}{\xLongleftarrow{\textcolor{black}{\mbox{\tiny BSD}}}}
 \xLongrightarrow{\mbox{\tiny Coates-Wiles}} \ 
L\,(1,\tilde{\chi})=0  \\
&\xLongleftrightarrow{\hskip 26pt \mbox{\tiny Asai} \hskip 25pt } \ 
\mathrm{egs}(\lambda)=0. 
\end{aligned}
\end{equation*}
\item 
In the proof of \,\((1)\Rightarrow(2)\), 
we show that \,\(\mathscr{E}_{\lambda}(\mathbf{Q}(\i))/[1+\i]\mathscr{E}_{\lambda}(\mathbf{Q}(\i))\)\, 
is generated by \,\((0,0)\)\, and at most one non-torsion point. 
Thus the \,\(\mathbf{Z}[\i]\)-rank of \,\(\mathscr{E}_{\lambda}(\mathbf{Q}(\i))\)\, is at most one. 
If \,\(\mathrm{egs}(\lambda)=0\), 
then \,\(\mathscr{E}_{\lambda}(\mathbf{Q}(\i))\)\, has \,\(\mathbf{Z}[\i]\)-rank one, 
that is, MW-rank two. 
\end{oitem}
}
\end{remark}

\begin{example}
\label{4817}{\rm
Take  \,\(\lambda=41+56\,\i\), \,\(\ell={\lambda}\overline{{\lambda}}=4817\equiv 1\bmod{16}\). 
Then \,\(\lambda=\alpha^4-\beta^2\), where
\begin{equation*}
\alpha=\frac{-7\i(2+\i)(4+\i)}
            {3(1+2\i)(2+3\i)}, \ \ \ 
\beta=\frac{-(1+\i)^5(3+2\i)(7+8\i)(6+11\i)}
           {3^2(1+2\i)^2(2+3\i)^2}
\end{equation*}
and \,\(P=(\alpha^2,\,\alpha\beta)\)\, is a point of infinite order. 
This is given by a calculation by \texttt{MAGMA}. 
It also says that the Mordell-Weil rank of \,\(\mathscr{E}_{{\lambda}}\)\, is \,\(2\). 
We know another ratinal point by \texttt{MAGMA} as follows. Let 
\begin{equation*}
\alpha'=\frac{\i(1+2\i)(2+3\i)}3, \ \ \ 
\beta'=\frac{\i\,7(1+\i)(2+\i)(4+\i)}{3^2}.
\end{equation*}
Then \,\(\lambda=-{\alpha'}^4+{\beta'}^2\i\)\, and 
\,\(Q=({\alpha'}^2\i,\,\alpha'\beta')\)\, is in  \,\(\mathscr{E}_{\lambda}(\mathbf{Q}(\i))\)\, 
and \,\(P=[1+\i]\,Q\). 
We do not know the point \,\(Q\)\, generates how much part of the MW-group. 
}
\end{example}
\vskip 10pt
\newpage
\section{Vanishing EGS and Kummer-type congruence\label{main_result}}\ 
\vskip 5pt
\noindent
We rewrite the expansion (\ref{coeff_D}) of \(\mathrm{cl}(u)\). 
Namely, we define \,\(G_n\)\, in \,\(\mathbf{Z}\)\, by 
\begin{equation*}
    \mathrm{cl}(u)
    =\sum_{n=0}^{\infty}G_{n}\,\frac{\,u^n\,}{n!}
    =1-2u^2+12\,\frac{u^4}{4!}-{216}\,\frac{u^6}{6!}+{7056}\,\frac{u^8}{8!}-{368928}\,\frac{u^{10}}{10!}+\cdots.
\end{equation*}
Of course \(G_n=n!\,{D_n}\). 
We denote by \,\(H_{\ell}\)\, the \textit{Hasse invariant} of \,\(y^2=x^3-x\)\, at \,\(\ell\ (\equiv1\bmod{4})\),
namely, 
\begin{equation*}
H_{\ell}=\lambda+\overline{\lambda}\equiv(-1)^{(\ell-1)/4}\,\binom{\,\frac{\ell-1}{2}\,}{\frac{\ell-1}{4}}\bmod{\ell}. 
\end{equation*}
Our main result is the following theorem. 
\vskip 5pt
\begin{theorem}\label{0212a}
The followings are equivalent\,{\rm:}
\begin{oitem}
\item \(\mathrm{egs}({\lambda})=0\)\,{\rm;}
\item \(\ell\ \big|\ {G}_{\frac34(\ell-1)}\)\,{\rm; \ (}This is a special case of {\rm (\ref{initial_d}).)}
\item\label{initial_d}
Let \(e\) be any positive integer satisfying \,\(e\equiv\frac34(\ell-1)\,\bmod{(\ell-1)}\). 
Then for any \(a\geq0\), 
\begin{equation*}
\sum_{r=0}^a\binom{\,a\,}{r}(-H_\ell)^{a-r}\frac{\,G_{e+r(\ell-1)}\,}{e+r(\ell-1)}
\equiv0\bmod{{\ell\,}^{a-\lfloor\frac{a}{\ell}\rfloor+1}}. 
\end{equation*}
\end{oitem}
\end{theorem}
\vskip 5pt
\noindent
In addition to this theorem, we have the following result. 
\vskip 5pt
\begin{theorem}\label{two_term_congr}
Assume that \,\(\mathrm{egs}(\lambda)=0\). 
Let \,\(b\ge 0\)\, be a given integer and \(e\) be any positive integer satisfying \,\(e\equiv\frac34(\ell-1)\,\bmod{(\ell-1)}\). 
Then the congruence 
\begin{equation}\label{0306f}
\frac{G_{e+k(\ell-1)}}{e+k(\ell-1)}
\equiv {H_{\ell}\,}^k\,{\cdot}\,\frac{\,G_e\,}{e} 
\bmod{\ell^{\,b+2}}
\end{equation}
holds for any \,\(k\)\, such that \,\({\ord}(k)\ge b\). 
\end{theorem}
\vskip 16pt
\section{\(\ell\)-adic explicit formula of an elliptic Gauss sum\label{proof_part1}}\ 
\vskip 5pt
\noindent
Recall our identification of \(\Z[\i]_{\lambda}\) and \(\Z_{\ell}\). 
As we treat a plenty of power series in \(\overline{\Z_{\ell}}[[x]]\) in this paper,
we summarize convention on notation here. 
Let \,\(f(x)\)\, and \,\(g(x)\)\, be power series in \,\(\overline{\Z_{\ell}}[[x]]\).  
For a rational number \,\(a\)\, in \,\(\Q\), we write
\begin{equation*}
f(x)\equiv g(x) \bmod{\lambda^a}
\end{equation*}
if all the coefficients of the terms in \,\(f(x)-g(x)\)\, have \,\(\ell\)-adic order at least \,\(a\). 
For a positive integer \,\(m\), we write
\begin{equation*}
f(x)\equiv g(x)\ \mathrm{mod\ deg}\,m
\end{equation*}
if \,\(f(x)-g(x)\)\, belongs to \,\(x^m\,\Z_{\ell}[[x]]\). 
Moreover, we write
\begin{equation*}
f(x)\equiv g(x)\ \mathrm{mod\ deg}\,m, \ \bmod{\,\lambda^a}
\end{equation*}
if all the coefficients of the terms of degree less than \,\(m\)\, in \,\(f(x)-g(x)\)\, 
have \,\(\ell\)-adic order at least \,\(a\). 
From now on, the number
\begin{equation*}
{d}=\frac34(\ell-1)
\end{equation*}
appears frequently. 
Taking a primitive \((\ell-1)\)-th root \(\zeta\) of \(1\) in \(\Z_\ell\), 
we define
\begin{equation}\label{def_Cl}
\mathrm{Cl}(u)=\frac12\sum_{j=0}^{\ell-1}\zeta^{-{d}j}\,\mathrm{cl}(\zeta^j u). 
\end{equation}
Then we have
\begin{equation*}
\mathrm{Cl}(u)=\frac{\ell-1}2\sum_{a=0}^\infty G_{{d}+a(\ell-1)}
\frac{u^{{d}+a(\ell-1)}}{({d}+a(\ell-1))!}. 
\end{equation*}
Thus \,\(G_{{d}+a(\ell-1)}/({d}+a(\ell-1))\)\, is 
the coefficient of the power series expansion of \,\((\ell-1)^{-1}\,\mathrm{Cl}(u)/u\)\, of 
\,\({u^n}/{n!}\)\, with \,\(n={d}+a(\ell-1)-1\). 
\par
For a proof of Theorem \ref{0212a}, 
we give an \,\(\ell\)-adic explicit formula of \,\(\mathrm{egs}(\lambda)\)\,
by using the Lubin-Tate formal group. 

Let \({\LT}(x,y)\) be 
the Lubin-Tate formal group over \,\(\Z_\ell\)\, corresponding 
to \,\(\lambda\)-plication \,\([\lambda]_{\LT}(x)=\lambda x+x^{\ell}\). 
Then non-zero points of the group \,\({\LT}[\lambda]\)\, of the \,\(\lambda\)-division points are roots of 
\,\(\lambda+x^{\ell-1}=0\). 
Let \,\(f_0(x)\)\, be the formal logarithm of \,\({\LT}(x,y)\). 
It follows from \,\([\lambda]_{\LT}(x)={f_0}^{-1}(\lambda f_0(x))\equiv x^{\ell}\bmod{\lambda}\)\,
that \,\(\lambda f_0(x)\equiv f_0(x^\ell) \bmod{\lambda}\)\, 
by Proposition 4.2 of Honda \cite{Ho}.
Namely, \({\LT}(x,y)\)\, is of type \,\(\lambda-T\). 
Let \,\({\sl}\,(x,y)\) be the formal group defined by 
\begin{equation*}
\smash{{\sl}\,(x,y)=\mathrm{sl}(\mathrm{sl}^{-1}(x)+\mathrm{sl}^{-1}(y))}
\end{equation*}
By the definition of \,\({\sl}\,(x,y)\), 
\(\lambda\)-plication \,\([\lambda]_{\sl}(x)\)\, satisfies 
\begin{equation*}
[\lambda]_{\sl}\,\circ \mathrm{sl}(x)=\mathrm{sl}(\lambda x).
\end{equation*}
Thus \,\(\varLambda\)\, is a points of 
the group \,\({\sl}\,[\lambda]\)\, of \,\(\lambda\)-division points. 
As is well-known (see for instance, Proposition 8.2 of Lemmermeyer \cite{Le} 
or Theorem 1.28 in \cite{O2} which gives another proof by usinig the relation \,\(\wp(u)=\mathrm{sl}(u)^{-2}\)\,), 
it holds that
\begin{equation*}
[\lambda]_{{\sl}}(x)\equiv x^\ell \bmod{\lambda}. 
\end{equation*}
Hence, the formal group \,\({\sl}\,(x,y)\)\, is of type \,\(\lambda-T\)\, as well. 
Since the formal group \,\({\LT}(x,y)\)\, is of the same type \,\(\lambda-T\), 
there exists the unique strong isomorphism  \,\(\iota\)\, over \,\(\Z_\ell\)\, from 
\({\LT}(x,y)\)\, to \,\({\sl}\,(x,y)\). 
Namely, there uniquely exists \,\(\iota(x)\)\, in \,\(\Z_{\ell}[[x]]\)\, such that
\begin{equation*}
\iota\big(\LT(x,y)\big)=\sl\,\big(\iota(x),\iota(y)\big), \ \ 
\iota(x)\equiv x\ \mbox{mod deg}\ {2}.
\end{equation*}
Then there exists \,\(\eta\)\, of the group \,\({\LT}[\lambda]\)\, 
of \,\(\lambda\)-division points of \,\({\LT}(x,y)\)\, such that 
\begin{equation*}
\varLambda=\varphi(1/\lambda)=\iota(\eta).
\end{equation*}
We recall that \,\(\eta^{\ell-1}=-\lambda\) (see (\ref{varLambda_lambda})). 
Since 
\begin{equation*}
\mathrm{cl}(u)=\phi\circ \mathrm{sl}(u),\ \ \mbox{where} \ \ \phi(x)=\smash{\sqrt{\frac{1-x^2}{1+x^2}}}, 
\end{equation*}
we have
\begin{equation*}
\psi(1/\lambda)=\phi\circ \iota(\eta).
\end{equation*}
We note that \,\(\phi(x)\)\, is in \,\(\Z_\ell[[x]]\). 
\par
Taking a primitive \,\((\ell-1)\)-th root \,\(\zeta\)\, of \,\(1\)\, in \,\(\Z_\ell\), 
we define 
\begin{equation*}
\mathrm{Sl}(u)=\frac{1}{4}\sum_{j=0}^{\ell-1}\zeta^{-dj}\,\mathrm{sl}(\zeta^j u). 
\end{equation*}
Then we have
\begin{equation*}
\mathrm{Sl}(u)=\frac{\ell-1}{4}\sum_{a=0}^\infty 
C_{d+a(\ell-1)}u^{d+a(\ell-1)}
\end{equation*}
as well as 
\begin{equation*}
\smash{
\mathrm{Cl}(u)=\frac{\ell-1}{2}\sum_{a=0}^\infty D_{d+a(\ell-1)}
u^{d+a(\ell-1)}.}
\end{equation*}
\vskip 10pt
\begin{lemma}\label{0213a2}
{\rm (1)} If \,\(\ell \equiv 5 \bmod{8}\), 
the equation \,\(\mathrm{egs}(\lambda)=(\mathrm{Sl}\circ f_0)(\eta)\)\, holds. \\
{\rm (2)} If \,\(\ell \equiv 1 \bmod{8}\), 
the equation \,\(\mathrm{egs}(\lambda)=(\mathrm{Cl}\circ f_0)(\eta)\)\, holds. 
\end{lemma}

\begin{proof}
It follows from \,\([\zeta]_{{\sl}}(x)=\mathrm{sl}(\zeta\mathrm{sl}^{-1}(x)) \in \Z_\ell[[x]]\)\,
that \(\mathrm{Sl}\circ \mathrm{sl}^{-1}(x)\)\, is in \,\(\Z_\ell [[x]]\), 
and from \,\(\mathrm{sl}^{-1}\circ \iota(x)=f_0(x)\)\, that 
\,\(\mathrm{Sl}\circ f_0(x)\)\, is also in \,\(\Z_\ell [[x]]\). 
Since \,\(\mathrm{cl}(x)=\phi\circ \mathrm{sl}(x)\)\, and 
\,\(\phi(x)\)\, is in \,\(\Z_\ell[[x]]\), 
we see \,\(\mathrm{cl}(\zeta \mathrm{sl}^{-1}(x))=\phi\circ [\zeta]_{{\sl}}(x)\).  
Hence, \(\mathrm{Cl}\circ f_0(x)\)\, is in \,\(\Z_\ell [[x]]\). 
For \,\(\alpha\)\, in \,\(\Z[\i]\)\, coprime to \,\(\lambda\), 
we convert, if necessary, \,\(\zeta\)\, another \((p-1)\)-th root of \(1\) as satisfying \,\(\alpha\equiv \zeta \bmod{\lambda}\). 
Then \,\(\varphi(\alpha/\lambda)=[\alpha]_{{\sl}}(\varLambda)=[\zeta]_{{\sl}}(\varLambda)\).
Since \,\(\varLambda=\iota(\eta)\)\, and \,\(\mathrm{sl}^{-1}\circ \iota(x)=f_0(x)\), 
we have \,\(\varphi(\alpha/\lambda)
=(\mathrm{sl}\circ \zeta \mathrm{sl}^{-1})(\varLambda)
=(\mathrm{sl}\circ \zeta f_0)(\eta)\) .
We also have  \,\(\psi(\alpha/\lambda)=(\mathrm{cl}\circ \zeta f_0)(\eta)\).
Since \,\(\chi_{\lambda}(\alpha)=\chi_\lambda(\zeta)=\zeta^{-d}\),
we have 
\begin{equation*}
\mathrm{egs}(\lambda)=
\frac{1}{4}\sum_{\alpha=1}^{\ell-1}\chi_{\lambda}(\alpha)\varphi(\alpha/\lambda)
=(\mathrm{Sl}\circ f_0)(\eta)
\end{equation*}
in the case of \,\(\ell \equiv 5 \bmod{8}\), and 
\begin{equation*}
\mathrm{egs}(\lambda)=\frac{1}{2}
\sum_{\alpha=1}^{\ell-1}\chi_{\lambda}(\alpha)\psi(\alpha/\lambda)
=(\mathrm{Cl}\circ f_0)(\eta)
\end{equation*}
in the case of \,\(\ell \equiv 1 \bmod{8}\). 
This completes the proof of Lemma \ref{0213a2}.
\end{proof}

\begin{lemma}\label{0214d}
{\rm (1)} \, If \(\ell \equiv 5 \bmod{8}\), it holds that
\begin{equation*}
\mathrm{egs}(\lambda)\equiv \frac{\ell-1}{4}
C_{{d}}\,\eta^{{d}}
\bmod{\eta^\ell}.
\end{equation*}
{\rm (2)} \, If \,\(\ell \equiv 1 \bmod{8}\), 
it holds that
\begin{equation*}
\mathrm{egs}(\lambda)\equiv 
\frac{\ell-1}{2}
D_{{d}}\,\eta^{{d}}\bmod{\eta^\ell}.
\end{equation*}
\end{lemma}

\begin{proof}
\!\!Since \,\(\lambda f_0(x)=f_0\circ [\lambda]_{{\LT}}=f(\lambda x+x^\ell)\),
we have \,\(\lambda f_0(x)\equiv f_0(\lambda x)\ \mathrm{mod\,deg}\ \ell\). 
Thus we have \,\(f_0(x) \equiv x\ \mathrm{mod\,deg}\ \ell\)\, and
\begin{equation*}
\mathrm{Sl}\circ f_0(x)\equiv \mathrm{Sl}(x)\equiv 
\frac{\ell-1}{2}C_{{d}}\, x^{{d}}\ 
\mathrm{mod\ deg}\ \ell.
\end{equation*}
Similarly we have 
\begin{equation*}
\mathrm{Cl}\circ f_0(x)\equiv \mathrm{Cl}(x)\equiv 
\frac{\ell-1}{4}D_{{d}}\,x^{{d}}\ 
\mathrm{mod\ deg}\ \ell.
\end{equation*}
Since \,\(\mathrm{Sl}\circ f_0(x)\)\, and 
\,\(\mathrm{Cl}\circ f_0(x)\)\, belong to  \,\(\Z_{\ell}[[x]]\), 
the assertion follows. 
\end{proof}
\textit{Proof of Theorem} \ref{cong_for_1_mod_8}. \ 
Because of \,\(\tilde{\lambda}=\gamma(S)^{-1}\smash{\prod\limits_{r\in S}\varphi(r/\lambda)}
\equiv \eta^{\frac{\ell-1}{4}} \bmod{\eta^{\frac{\ell-1}{4}+1}}\)\, 
and \(\mathrm{egs}(\lambda)=A_\lambda \tilde{\lambda}^3\), we have
\begin{equation*}
\smash{
A_\lambda
\equiv 
-\frac{1}{4}
C_{{d}}\,\bmod{\eta^{\frac{\ell-1}{4}+1}}
}
\end{equation*}
by using Lemma \ref{0214d}. 
Since both sides are rational number, 
we have the assertion of Theorem \ref{cong_for_1_mod_8}. 
\qed

%
%
%
%
%
\vskip 10pt
\section{Application of the Hochschild formula}\ 
\vskip 5pt
\noindent
In this section, we use the following formula known as the \textit{Hochschild formula}.  
For a proof of this formula, see Matsumura \cite{Ma}, p.197, Theorem 25.5.   
\begin{lemma}\label{Hochschild}
Let \,\(R\)\, be a commutative ring of characteristic \,\(\ell\). 
Let \,\(\delta\)\, be a derivation over \,\(R\). 
Then, for any element \,\(b\)\, in \,\(R\), we have 
\begin{equation*}
(b\delta)^\ell=b^\ell \delta^\ell+((b\delta)^{\ell-1}(b))\cdot{\delta}.
\end{equation*}
\end{lemma}

We put \,\(u=f_0(x)\). 
By the definition of \,\(H_{\ell}\), we have 
\((\overline{\lambda}-T)(\lambda-T)=\ell-H_\ell T+T^2\). 
Since \,\(\overline{\lambda}-T\)\, is a unit in \,\(\Z_{\ell}[[T]]\), 
any formal group over \,\(\Z_{\ell}\) of type \,\(\lambda-T\) 
is also of type \,\(\ell-H_\ell T+T^2\). 

\begin{lemma}\label{0216c2}
Let \(\phi(x)\) be a power series in \(\Z_\ell [[x]]\). 
Then 
\begin{equation}\label{0214b2}
\bigg(\Big(\frac{d}{du}\Big)^{\ell}-H_{\ell}\frac{d}{du}\bigg)
\phi(x) \in \ell \Z_\ell[[x]].
\end{equation}
\end{lemma}

\begin{proof}
Since \(du/dx={f_0}^{\prime}(x)\) is in \(\Z_\ell[[x]]^{\times}\), 
\(\frac{d}{du}=\frac{dx}{du}\,\frac{d}{dx}\)\, 
is a derivation on \,\(\Z_{\ell}[[x]]\).
Since \,\({\LT}(x,y)\)\, is of type \,\(\ell-H_\ell T+T^2\), 
there exists \,\(h(x)\)\, in \,\(\Z_{\ell}[[x]]\)\, such that 
\begin{equation*}
{\ell}f_0(x)-H_{\ell}f_0(x^\ell)+f_0(x^{\ell^2})={\ell}h(x). 
\end{equation*}
This yields that
\begin{equation*}
{f_0}^\prime (x)-H_\ell {f_0}^\prime(x^\ell)x^{\ell-1}
\equiv h^\prime(x) \bmod{\ell}. 
\end{equation*}
Differentiating this \,\(\ell\,{-}\,1\)\, times by \,\(x\), 
we have 
\begin{equation*}
{f_0}^{(\ell)}(x)-H_\ell {f_0}^\prime(x^\ell)(\ell-1)!
\equiv h^{(\ell)}(x) \bmod{\ell}. 
\end{equation*}
By \,\((\ell-1)!\equiv -1 \bmod{\ell}\), 
\({f_0}^\prime(x)\)\, in \,\(\Z_\ell [[x]]\), and 
\(h^{(\ell)}(x)\equiv 0 \bmod{\ell}\), we have
\begin{equation}\label{f_0_relation}
f_0^{(\ell)}(x)+H_\ell ({f_0}^\prime(x))^\ell
\equiv 0 \bmod{\ell}. 
\end{equation}
Let \,\(\phi(x)\)\, be a power series in \,\(\Z_\ell [[x]]\). 
\begin{equation*}
0\equiv \bigg(\frac{d}{dx}\bigg)^{\!\ell}\phi(x) \equiv 
\left(\frac{du}{dx}\frac{d}{du}\right)^\ell \phi(x) \bmod{\ell}.
\end{equation*}
By using the Hochschild formula (Lemma \ref{Hochschild}), we have
\begin{equation*}
0\equiv 
\left(\frac{du}{dx}\right)^\ell 
\frac{d^\ell\phi}{du^\ell}+
\left(\frac{du}{dx}\frac{d}{du}\right)^{\ell-1}\frac{du}{dx}\cdot 
\frac{d\phi}{du} 
\equiv 
\left(\frac{du}{dx}\right)^\ell 
\frac{d^\ell\phi}{du^\ell}+
\frac{d^\ell u}{dx^\ell}\cdot 
\frac{d\phi}{du}  \bmod{\ell}.
\end{equation*}
Thus we have 
\begin{equation*}
\frac{d^\ell\phi}{du^\ell}+
\left(\frac{du}{dx}\right)^{-\ell} 
\frac{d^\ell u}{dx^\ell}\cdot 
\frac{d\phi}{du} 
\equiv 0 \bmod{\ell}.
\end{equation*}
By (\ref{f_0_relation}) 
we have 
\begin{equation}\label{0305a}
\frac{d^\ell\phi}{du^\ell}-H_\ell
\frac{d\phi}{du} 
\equiv 0 
\bmod{\ell\,\Z_{\ell}[[x]]}.
\end{equation}
This completes the proof of Lemma \ref{0216c2}.
\end{proof}

Let \,\(c\)\, be an integer. 
Let \,\(\phi(x)\)\, be any element in \,\(\ell^c\,\Z_\ell[[x]]\). 
We define the expansion of \,\(\phi\circ \mathrm{sl}(u)\)\, by 
\begin{equation*}
\smash{
\phi\circ \mathrm{sl}(u)=\sum_{k\ge 0}\frac{b_k}{k!}u^k\ \ \ (b_k\in\Q_\ell).
}
\end{equation*}
We denote
\begin{equation*}
\varOmega_{\ell}=\bigg(\frac{d}{du}\bigg)^{\ell}-H_\ell \frac{d}{du}. 
\end{equation*}
For any non-negative integer \(a\), 
we see \,\({\varOmega_{\ell}}^a\,\phi(x)\)\, in \,\(\ell^{a+c}\,\Z_{\ell}[[x]]\)\, by (\ref{0305a}).
Since 
\begin{equation*}
{\varOmega_{\ell}}^a\bigg(\!\sum_{k\ge 0}\frac{b_k}{k!}u^k\!\bigg)
=\!\sum_{k\ge 0}\!
\left(
\sum_{r=0}^a\!
\binom{\color{black}a}{r}
(-H_{\ell})^{a-r}\,b_{k+a+r(p-1)}
\!\right)\!
\frac{u^k}{k!} 
\in \ell^{a+c}\,\Z_{\ell}[[x]]
\subset\ell^{a+c}\,\Z_{\ell}\llangle{u}\rrangle,
\end{equation*}
we have 
\begin{equation}
\label{0216d}
\sum_{r=0}^a
\binom{\color{black}a}{r}
(-H_\ell)^{a-r}\,b_{k+a+r(p-1)}
\equiv 0 \bmod{\ell^{a+c}}. 
\end{equation}
\vskip 10pt
\section{Proof of the main theorem}\ 
\vskip 5pt
\noindent
We prove the implications (1) \(\Rightarrow\) (\ref{initial_d}) \(\Rightarrow\)  (2) in Theorem \ref{0212a}. 

\begin{lemma}\label{Cl_over_lambda_x}
If \,\(\mathrm{egs}(\lambda)=0\), then 
\,\((\mathrm{Cl}\circ f_0)(x)/(\lambda x+x^\ell)\)\, is in \,\(\Z_\ell[[x]]\). 
\end{lemma}
\begin{proof}
Assume \,\(\mathrm{egs}(\lambda)=0\) and put \,\((\mathrm{Cl}\circ f_0)(x)=\sum_{n=0}^{\infty} b_n x^n\)\, with \,\(b_n\)\, in \,\(\Z_\ell\).  
Then, \((\mathrm{Cl}\circ f_0)(\eta)=\sum_{n=0}^{\infty}b_n \eta^n=0\)\, by Lemma \ref{0213a2}. 
Therefore, 
\begin{equation*}
(\mathrm{Cl}\circ f_0)(x)
=\sum_{n=0}^{\infty}b_n x^n-\sum_{n=0}^{\infty}b_n \eta^n
=(x-\eta)\sum_{n=1}^{\infty} b_n \frac{x^n-\eta^n}{x-\eta}\in \Z_\ell[\eta][[x]]
\end{equation*}
because \,\(x-\eta\)\, divides \,\(x^n-\eta^n\). 
Similarly, any conjugate of \,\(x-\eta\)\, divides \,\((\mathrm{Cl}\circ f_0)(x)\)\, 
and \,\(x\)\, divides \,\((\mathrm{Cl}\circ f_0)(x)\). 
Hence, the assertion follows. 
\end{proof}

\begin{lemma}\label{0214c2}
Let \,\(\nu\)\, be a positive integer. 
Assume  \,\(\mathrm{egs}(\lambda)=0\). 
If \,\(n<\nu\ell\,(\ell-1)\), then the coefficient 
in \,\(\lambda^{\nu-1}(\mathrm{Cl}\circ f_0)(x)/f_0(x)\)\,
of \,\(x^n\)\, belongs to \,\(\ell\,\Z_{\ell}\).
\end{lemma}
\begin{proof}
Since \,\({f_0}^{\prime}(x)\)\, is in \,\(\Z_\ell [[x]]\), 
it is seen that \,\(\xi^{-1}f_0(\xi x)\)\, is in \,\(\Z_{\ell}[\xi][[x]]\)\, 
for any \,\(\ell\)-adic algebraic integer \,\(\xi\)\, with \,\({\ord}(\xi)=1/(\ell-1)\)\,
by calculating the \,\(\ell\)-adic order of each coefficients of its expansion. 
We put 
\begin{equation*}
g(x)=\frac{\lambda x+x^\ell}{\lambda f_0(x)}
=\frac{\lambda x+x^{\ell}}{f_0(\lambda x+x^{\ell})}.
\end{equation*}
Then 
\begin{equation*}
\smash{
g(\xi^{\frac{1}{\ell}}x)=
\frac{\xi(\lambda \xi^{\frac{1}{\ell}-1}x+x^\ell)}
     {f_0(\xi(\lambda \xi^{\frac{1}{\ell}-1}x+x^\ell))}
\in \Z_{\ell}[\xi^{\frac{1}{\ell}}][[x]], }
\end{equation*}
because 
\,\({\ord}(\lambda \xi^{\frac{1}{\ell}-1})=1-1/\ell\). 
Thus the \,\(\ell\)-adic order of the coefficient of \,\(x^n\)\, of \,\(g(x)\)\, 
is greater than or equal to \,\(-\big\lfloor\frac{n}{\ell(\ell-1)}\big\rfloor\). 
Therefore, we see
\begin{equation*}
\lambda^\nu g(x)\equiv 0\ \,\mathrm{mod\ deg}\ \nu\ell(\ell-1),\ \bmod{\lambda}.
\end{equation*}
Thus, each coefficient of the terms of degree less than \,\(\nu\ell(\ell-1)\)\, in 
\begin{equation*}
\lambda^{\nu-1} \frac{\mathrm{Cl}(u)}{u}
=\lambda^{\nu-1}\,\frac{\mathrm{Cl}\circ f_0(x)}{f_0(x)}
=\frac{\mathrm{Cl}\circ f_0(x)}{\lambda x+x^\ell}
\cdot \lambda^{\nu}g(x) 
\end{equation*}
is in \,\(\ell\,\Z_{\ell}\)\, by Lemma \ref{Cl_over_lambda_x}.
\end{proof}

\begin{lemma}\label{0305b}
Assume that \,\(\mathrm{egs}(\lambda)=0\). 
Then we have for \,\(a<\nu\ell\)\, that
\begin{equation*}
\sum_{r=0}^a
\binom{n}{r}
(-H_\ell)^{a-r}
\frac{G_{d+r(\ell-1)}}
{d+r(\ell-1)}
\equiv 0 \bmod{\ell^{a-\nu+2}}.
\end{equation*}
\end{lemma}

\begin{proof}
We denote by \,\(\phi(x)\)\, 
the sum of the terms in \,\(\lambda^{\nu-1}(\mathrm{Cl}\circ f_0(x))/f_0(x)\) 
of degree less than \,\(\nu\ell (\ell-1)\). 
The assertion of Lemma \ref{0214c2} is amount to the same thing to say 
\,\(\phi(x)\)\, in \,\(\ell\,\Z_\ell[[x]]\). 
Now, the last argument in the previous section is applied by 
plugging \,\(c=1\)\, and \,\(b_k=\lambda^{\nu-1}G_k/k\)\, for \,\(k<\nu\ell (\ell-1)\). 
By using (\ref{0216d}), we have 
\begin{equation*}
\lambda^{\nu-1}
\sum_{r=0}^a\binom{n}{r}
(-H_\ell)^{a-r}
\frac{G_{d+r(\ell-1)}}
{d+r(\ell-1)}
\equiv 0 \bmod{\ell^{a+1}}
\end{equation*}
for \,\(d+a(\ell-1)<\nu\ell(\ell-1)\), that is, for \,\(a<\nu\ell\), hence the desired congruence. 
\end{proof}
We take \,\(\nu=\lfloor{a/\ell}\rfloor+1\)\, in Lemma \ref{0305b}. 
Then \,\(\nu\)\, satisfies \,\(a<\nu\ell\)\, for any \,\(a\ge 0\).  
Therefore, we conclude that
\begin{equation}\label{0305d}
\sum_{r=0}^a
\binom{a}{r}
(-H_\ell)^{a-r}
\frac{G_{d+r(\ell-1)}}
{d+r(\ell-1)}
\equiv 0 \bmod{\ell^{\,a-\lfloor\frac{a}{\ell}\rfloor+1}}
\end{equation}
for any \,\(a\ge 0\). 
By expanding \,\(\{(x+1)-1\}^{k}(x+1)^a\), we have 
\begin{equation}\label{0305f}
\sum_{j=0}^{k}\sum_{m=0}^{j+a}
(-1)^{k-j}\binom{k}{j}\binom{j+a}{m}x^m
=\sum_{r=0}^a\binom{a}{r}\,x^{r+k}
\end{equation}
for any \,\(k\geq 0\). 
By using the structure of (\ref{0305f}), 
we simplify a linear combination of the sum of the left hand side of (\ref{0305d}) 
for various \,\(a\)'s as follows:
\begin{align*}
&\sum_{j=0}^{k}
(-1)^{k-j}
(-H_\ell)^{k-j}
\binom{k}{j}
\sum_{m=0}^{j+a}
\binom{j+a}{m}
(-H_\ell)^{j+a-m}\frac{G_{d+m(\ell-1)}}{d+m(\ell-1)}\\
&=\sum_{j=0}^{k}
  \sum_{m=0}^{j+a}
(-1)^{k-j}
\binom{k}{j}
\binom{j+a}{m}
(-H_\ell)^{k+a-m}
\frac{G_{d+m(\ell-1)}}{d+m(\ell-1)} \\
&=\sum_{r=0}^a
\binom{a}{r}
(-H_\ell)^{a-r}
\frac{G_{d+(r+k)(\ell-1)}}{d+(r+k)(\ell-1)}.
\end{align*}
Since the exponent \,\(a-\lfloor{a}/{\ell}\rfloor+1\)\, of the modulus 
in (\ref{0305d}) is a monotone increasing function on \,\(a\), we have
\begin{equation*}
\sum_{r=0}^a
\binom{a}{r}(-H_\ell)^{a-r}
\frac{G_{d+(r+k)(\ell-1)}}{d+(r+k)(\ell-1)}
\equiv 0
\bmod{\ell^{\,a-\lfloor\frac{a}{\ell}\rfloor+1}}.
\end{equation*}
This is no other than the assertion (3) of Theorem \ref{0212a}. 
Thus (1) implies (3). 
Plugging \,\(a=k=0\), we have 
\begin{equation*}
\smash{\frac{G_d}{d}\equiv 0 \bmod{\ell}.} 
\end{equation*}
Thus (3) implies (2) in Theorem \ref{0212a}.
\vskip 15pt
\section{The two term congruence}\ 
\label{toward_p-adic_L}
\vskip 5pt
\noindent
In this section, we show Theorem \ref{two_term_congr}. 

\begin{proof}[Proof of Theorem {\rm\ref{two_term_congr}}]
We show (\ref{0306f}) by using the induction on \,\(b\). 
\\ 
(i) \,When \,\(b=0\), the assertion follows from Theorem \ref{0212a}. 
Indeed, by taking \,\(a=1\) in (3) of Theorem \ref{0212a}, we have 
\begin{equation*}
\frac{G_{e+(\ell-1)}}{e+(\ell-1)}
\equiv H_\ell 
\frac{G_e}{e} 
\bmod{\ell^{2}}
\end{equation*}
{\color{black}for any \,\(e\)\, satisfying the assmption in Theorem \ref{two_term_congr}.} 
This is the case of \,\(b=0\)\, of (\ref{0306f}). 
\\ 
(ii) \,We assume that there exists some integer \,\(c>0\)\, such that 
(\ref{0306f}) holds for any \,\(b< c\). 
Then, by taking \,\(a=\ell^c\)\, in (3) of Theorem \ref{0212a}, we have 
\begin{equation}\label{0306c}
\sum_{r=0}^{\ell^c} \binom{\ell^c}{r}
(-H_\ell)^{\ell^c -r}
\frac{G_{e+r(\ell-1)}}{e+r(\ell-1)}\equiv 0 \bmod{\ell^{\ell^c-\ell^{c-1}+1}}.
\end{equation}
For \,\(1\le r\le \ell^c-1\), we have 
\begin{align*}
&\binom{\ell^c}{r}(-H_\ell)^{\ell^c -r}
\frac{G_{e+r(\ell-1)}}{e+r(\ell-1)}
+\binom{\ell^c}{\ell^c-r}
(-H_\ell)^{r}\frac{G_{e+(\ell^c -r)(\ell-1)}}{e+(\ell^c -r)(\ell-1)}\\
&=\binom{\ell^c}{r}
\left(
(-H_\ell)^{\ell^c -r}\frac{G_{e+r(\ell-1)}}{e+r(\ell-1)}
+(-H_\ell)^{r}\frac{G_{e+(\ell^c -r)(\ell-1)}}{e+(\ell^c -r)(\ell-1)}
\right).
\end{align*}
Using the Legendre's formula which gives the exact \,\(\ell\)-adic order for 
the factorial of any positive integer, 
it is easily proved that (see also Dickson \cite{Di}, p.270) 
\begin{equation*}
{\ord}\Big(\binom{\ell^{\,c}}{r}\Big)
=c-{\ord}(r)
\end{equation*}
provided that \,\(1\le r \le \ell^c-1\). 
We note that \,\({\ord}(r)={\ord}(\ell^c-r)<c\)\, for \,\(1\le r \le \ell^c-1\).
Then we have
\begin{align*}
(-H_\ell)^{\ell^c -r}
&\frac{G_{e+r(\ell-1)}}{e+r(\ell-1)}
+(-H_\ell)^{r}\frac{G_{e+r(\ell-1)}}{e+(\ell^c-r)(\ell-1)}\\
&\equiv (-1)^{\ell^c -r}{H_{\ell}}^{\ell^c}
\frac{G_{e}}{e}+(-1)^r{H_{\ell}}^{\ell^c}
\frac{G_{e}}{e}
\bmod{\ell^{2+{\ord}(r)}}\\
&\equiv \{(-1)^{\ell^c -2r}+1\}
(-1)^r{H_{\ell}}^{\ell^c}\frac{G_{e}}{e}
\equiv 0\bmod{\ell^{2+{\ord}(r)}}
\end{align*}
by using the assumption of the induction. 
As \,\((c-{\ord}(r))+(2+{\ord}(r))=2+c\)\, we see
\begin{equation*}
\binom{\ell^c}{t}(-H_\ell)^{\ell^c -r}
\frac{G_{e+r(\ell-1)}}{e+r(\ell-1)}
+\binom{\ell^c}{\ell^c-r}
(-H_\ell)^r\frac{G_{e+(\ell^c-r)(\ell-1)}}{e+(\ell^c-r)(\ell-1)}
\equiv 0\bmod{\ell^{2+c}}.
\end{equation*}
By taking the summation on  \,\(r\)\, such that \,\(1\le r \le (\ell^c-1)/2\), 
we have 
\begin{equation}\label{0306d}
\sum_{r=1}^{\ell^c-1}\binom{\ell^c}{r}
(-H_\ell)^{\ell^c -r}
\frac{G_{e+r(\ell-1)}}
{e+r(\ell-1)}
\equiv 0 \bmod{\ell^{2+c}}.
\end{equation}
Since \,\(\ell^c-\ell^{c-1}+1\ge 2+c\)\, 
holds for \,\(c\ge 0\), we have 
\begin{equation*}
\frac{G_{e+\ell^c(\ell-1)}}{e+\ell^c(\ell-1)}
\equiv {H_{\ell}}^{\ell^c}\frac{G_e}{e}
\bmod{\ell^{2+c}}
\end{equation*}
for any \,\(e\ge 1\)\, by (\ref{0306c}) and (\ref{0306d}). 
Thus the assertion (\ref{0306f}) follows in the case of \,\(b=c\). 
\\ 
(iii) \,By the induction, (\ref{0306f}) holds for any \,\(b\ge 0\). 
\end{proof}
\vskip 10pt
\begin{remark}{\rm%
(1) \,On the classical Bernoulli numbers, if \(b\leq{d}\), then
\begin{equation*}\label{two_term_Kummer}
\frac{B_d}{d}\equiv\frac{B_{d+kp^{b-1}(p-1)}}{d+kp^{b-1}(p-1)}\bmod{p^b}\ \ \ 
\end{equation*}
Here the condition \(b\leq{d}\) is essential. 
However, for any \(b\), \(d\),  and \(k\), it is known that 
\begin{equation}\label{Iwasawa}
(1-p^{d-1})\frac{B_d}{d}\equiv\big(1-p^{d+kp^{b-1}(p-1)-1}\big)\frac{B_{d+kp^{b-1}(p-1)}}{d+kp^{b-1}(p-1)}\bmod{p^b}.
\end{equation}
Of course, the extra factors are no other than Euler \(p\)-factors of the Riemann \(\zeta\)-function. 
This consideration suggests that the reason why Theorem \ref{two_term_congr} holds 
without condition on \(b\) and \(d\) is that the Euler 
\(\lambda\)-factor of the Hecke \(L\)-function for \,\(\mathscr{E}_{\pm\lambda}\)\, is \(1\). \\
(2) \,
On Kubota-Leopoldt \(p\)-adic \(L\)-function, it is fundamental that 
the special values of the corresponding complex \(L\)-function is given by 
(generalized) Bernoulli numbers and they satisfy (\ref{Iwasawa}) 
involving Euler \(p\)-factor of the complex \(L\)-series. 
However the congruence \ref{two_term_congr} is a relation on the numbers 
which are not exactly the special values but only their residues modulo some power 
of \,\(\ell\). 
}
\end{remark}
\vskip 15pt
\section{Central value of the Hecke \(L\)-function}\ 
\label{central_value}
\vskip 5pt
\noindent
In this section we refer to Koblitz \cite{K}. 
We modify \S 5 and \S 6 of Chapter 2 in \cite{K}. 

Put \,\(\mathcal{O}=\Z[\i]\)\, and take \,\(\beta\)\, in \,\(\mathcal{O}\). 
Let \,\(\widetilde{\chi}\)\, be 
a Hecke character of modulus \,\((\beta)\)\, of weight one. 
Namely, \,\(\widetilde{\chi}((\nu))=\chi_1(\nu)\overline{\nu}\), 
where \,\(\chi_1\)\, is a character 
form \,\((\mathcal{O}/(\beta))^{\times}\)\, to \,\(\C^{\times}\)
satisfying \,\(\chi_1(\i)=\i\). 
We define the Hecke \(L\)-function by 
\begin{equation*}
L(s,\widetilde{\chi})=
\sum_{\mathfrak{a}}\frac{\widetilde{\chi}(\mathfrak{a})}{N\mathfrak{a}^s}
=\frac{1}{4}\sum_{\nu \in \mathcal{O}}
\frac{\chi_1(\nu)\overline{\nu}}{|\,\nu\,|^{2s}}
=\frac{1}{4}
\sum_{\gamma\,{\rm mod}\,\beta}\chi_1(\gamma)
\sum_{{\alpha} \in \mathcal{O}}\frac{\overline{\gamma+{\alpha}\beta}}{|\,\gamma+{\alpha}\beta\,|^{2s}},
\end{equation*}
where \,\(\mathfrak{a}\)\, runs over the non-zero integral ideals of \(\mathcal{O}\) 
and \(N\mathfrak{a}=\cardinarity{\,\mathcal{O}/\mathfrak{a}}\) is the norm of \(\mathfrak{a}\). 

We use a method to get the following classically known fact.  
\begin{theorem}\label{0306a}
The function defined by
\begin{equation*}
\Lambda(s,\widetilde{\chi})=
\bigg(\frac{2\pi}{\sqrt{4N(\beta)}}\bigg)^{\!\!-s}\Gamma(s)\,L(s,\widetilde{\chi})
\end{equation*}
satisfies
\vskip -15pt
\begin{equation*}
\Lambda(s,\widetilde{\chi})=C(\widetilde{\chi})\,
\Lambda(2-s,\overline{\widetilde{\chi}}), 
\end{equation*}
where \ \(\smash{C(\widetilde{\chi})=-\i\,\beta^{-1}\!\!\!\sum\limits_{\lambda\,\bmod{\beta}}
\chi_1(\lambda)\,e^{2\pi\i\,\mathrm{Re}(\lambda/\beta)}}\).
\end{theorem}
\vskip 5pt
We do not need the result above itself but the following bi-product of its proof. 
\begin{lemma}\label{Koblitz_type}
We have the estimation
\begin{equation*}
L(1,\widetilde{\chi})<\frac4{\exp\big(\frac{\pi}{|\beta|}\big)-1}.
\end{equation*}
\end{lemma}
\vskip 5pt
\begin{proof}
First of all, we note that \,\(\big({2\pi}/{\sqrt{4N(\beta)}}\,\big)^{\!-s}\!=|\,\beta\,|^s\,\pi^{-s}\)\, 
because of \,\(|\,\beta\,|=\sqrt{N(\beta)}\). 
We give an outline of proof which is divided into four steps. 
\\
(Step 1) \  We define 
\begin{equation*}
F(t,\widetilde{\chi})=\frac{1}{4}\sum_{\nu \in \mathcal{O}}
\chi_1(\nu)\overline{\nu}e^{-\pi t|\,\nu\,|^2}. 
\end{equation*}
Then,
by using \,\(\smash{\displaystyle\int_0^\infty e^{-ct}t^s\frac{dt}{t}=c^{-s}\Gamma(s)}\),  we have 
\begin{equation}\label{integral_expression}
\pi^{-s}\Gamma(s)L(s,\widetilde{\chi})
=\int_0^\infty F(t,\widetilde{\chi})t^s\frac{dt}{t}. 
\end{equation}
(Step 2) \ For \,\(u\)\, in \(\mathbf{R}^2\)\, and \,\(w\)\, in  \(\mathbf{C}^2\), 
we define 
\,\(\theta_u(t)=\sum_{m\in\mathbf{Z}^2}(m+u)\cdot{w}\,e^{-\pi t|\,m+u\,|^2}\)\, 
and \,\(\theta^u(t)=\sum_{m\in\mathbf{Z}^2}m\cdot{w}\,e^{2\pi i m\cdot u}e^{-\pi t|\, m\,|^2}\), 
where \(\cdot\) stands for the inner product, holds (cf. \cite{K}, p.85, (5.16)). 
Then, 
\begin{equation}
\theta_u(t)=-\frac{i}{t^2}\,\theta^{u}\Big(\frac{1}{t}\Big), 
\label{0306a2}
\end{equation}
Applying (\ref{0306a2}), \(F(t,\widetilde{\chi})\)\, is rewritten 
\begin{equation*}
F(t,\widetilde{\chi})
=\frac{1}{4}\!\sum_{\gamma\,{\rm mod}\,\beta}\!\!\chi_1(\gamma)
\sum_{{\alpha}\in \mathcal{O}}
\overline{\beta{\alpha}{+}\gamma}\,e^{-\pi{}t\,|\beta{\alpha}{+}\gamma|^2}
=\frac{\overline{\beta}}{4}\!\sum_{\gamma\,{\rm mod}\,\beta}
\!\!\chi_1(\gamma)
\sum_{{\alpha}\in \mathcal{O}}
\overline{{\alpha}{+}\frac{\gamma}{\beta}}\,
e^{-\pi t|\beta|^2|{\alpha}{+}\frac{\gamma}{\beta}|^2}
\end{equation*}
By using a vector in \,\(\R^2\), 
we write the inner sum. 
For a given \,\(\gamma\), we put 
\,\(\frac{\gamma}{\beta}=u_1+u_2\i\)\, with \,\(u=(u_1,u_2)\)\, in \,\(\Q^2\), 
\({\alpha}=m_1+m_2\i\)\, with \,\(m=(m_1,m_2)\)\, in \,\(\Z^2\). 
Then we have 
\begin{align*}
\sum_{{\alpha}\in \mathcal{O}}
\bigg(\,\overline{{\alpha}+\frac{\gamma}{\beta}}\,\bigg)
e^{-\pi{t}\,|\,\beta\,|^2|\,{\alpha}+\frac{\gamma}{\beta}\,|^2}
&=\sum_{m\in \Z^2}
\overline{(m+u)\,{\cdot}\,(1,i)}\,
e^{-\pi t|\,\beta\,|^2|\,m+u\,|^2}\\
&=\sum_{m\in \Z^2}
(m+u)\,{\cdot}\,(1,-i)\,
e^{-\pi t|\,\beta\,|^2|\,m+u\,|^2}. 
\end{align*}
We put \(w=(1,-i)\). 
Then we have 
\begin{equation*}
F(t,\widetilde{\chi})
=\frac{\overline{\beta}}{4}\sum_{\gamma\,{\rm mod}\,\beta}
\chi_1(\gamma)\theta_u(|\,\beta\,|^2t). 
\end{equation*}
We note \(u\) depends on \(\gamma\). 
\\
(Step 3) \ 
By using the functional equation (\ref{0306a2}), 
we prove that of \,\(F(t,\widetilde{\chi})\).
\begin{equation*}
F\Big(\frac{1}{|\,\beta\,|^2 t},\widetilde{\chi}\Big)=
\frac{\overline{\beta}}{4}\sum_{\gamma\,{\rm mod}\,\beta}\chi_1(\gamma)\theta_u(\frac{1}{t})
=\frac{\overline{\beta}}{4}\sum_{\gamma\,{\rm mod}\,\beta}\chi_1(\gamma)(-it^2)\theta^u(t).
\end{equation*}
By the definition of \,\(\theta^u(t)\), 
we calculate the right hand side. 
\begin{align*}
F\Big(\frac{1}{|\,\beta\,|^2 t},\widetilde{\chi}\Big)
&=\frac{\overline{\beta}}{4}\sum_{\gamma\,{\rm mod}\,\beta}
\chi_1(\gamma)(-it^2)\sum_{m\in \Z^2}
m\cdot (1,-i) e^{2\pi i m\cdot u}e^{-\pi t |\, m\,|^2}\\
&=\frac{\overline{\beta}}{4}(-it^2)
\sum_{m\in \Z^2}m\cdot (1,-i) e^{-\pi t |\, m\,|^2}
\sum_{\gamma\,{\rm mod}\,\beta}
\chi_1(\gamma)
e^{2\pi i m\cdot u}. 
\end{align*}
It follows from \,\(m\cdot{u}=m_1u_1+m_2u_2=\mathrm{Re}\big((m_1-m_2 i)(u_1+u_2i)\big)
=\mathrm{Re}\big(\overline{{\alpha}}\,\frac{\gamma}{\beta}\big)\)\,
that the sum inside is essentially Gauss sum and
\begin{align*}
\sum_{\gamma\,\bmod{\beta}}\chi_1(\gamma)\,e^{2\pi i m\cdot u}
&=\overline{\chi_1(\overline{{\alpha}})}
\sum_{\gamma\bmod{\beta}}
\chi_1(\overline{{\alpha}}\gamma)\,e^{2\pi i {\rm Re}(\overline{{\alpha}}\gamma/\beta)}\\
&=\overline{\chi_1(\overline{{\alpha}})}
 \sum_{\gamma\,{\rm mod}\,\beta}
 \chi_1(\gamma)\,e^{2\pi i {\rm Re}(\gamma/\beta)}
 =\overline{\chi_1(\overline{{\alpha}})}\,i \beta\, C(\widetilde{\chi}).
\end{align*}
Since 
\,\(\smash{%
\displaystyle\sum_{m\in \Z^2}
m\,{\cdot}\,(1,-i)\,e^{-\pi t |\, m\,|^2}
=\displaystyle\sum_{{\alpha} \in \mathcal{O}}\overline{{\alpha}}e^{-\pi t |\, \overline{{\alpha}}\,|^2}
=\displaystyle\sum_{{\alpha} \in \mathcal{O}}{\alpha} e^{-\pi t |\, {\alpha}\,|^2}
}\)\, 
and \,\(\overline{\widetilde{\chi}}(\nu)=\overline{\chi_1}(\nu)\nu\), 
\vskip 0pt
\begin{equation}\label{0306b}
F\Big(\frac{1}{|\,\beta\,|^2 t},{\color{black}\widetilde{\chi}}\Big)
=\frac{\overline{\beta}\beta}{4}t^2C(\widetilde{\chi})
\sum_{{\alpha} \in \mathcal{O}}
\overline{\chi_1(\overline{{\alpha}})}
{\alpha} e^{-\pi t |\, {\alpha}\,|^2}
=|\,\beta\,|^2t^2C(\widetilde{\chi})F(t,\overline{\widetilde{\chi}}). 
\end{equation}
(Step 4) \ 
From (\ref{integral_expression}) we have
\begin{equation*}
\pi^{-s}\,\Gamma(s)\,L(s,\widetilde{\chi})=
\int_0^\infty t^s\,F(t,\widetilde{\chi})\frac{dt}{t}
=
\int_0^{\frac{1}{|\beta|}}\!t^s F(t,\widetilde{\chi})\frac{dt}{t}
+\int_{\frac{1}{|\beta|}}^\infty t^s F(t,\widetilde{\chi})\frac{dt}{t},
\end{equation*}
in which the former integration is rewritten as
\begin{equation*}
\int_0^{\frac{1}{|\beta|}} t^s F(t,\widetilde{\chi})\frac{dt}{t}
=|\,\beta\,|^{-2s}\!\!
\int^\infty_{\frac{1}{|\beta|}}
v^{-s} F\Big(\frac{1}{|\beta|^2v},\widetilde{\chi}\Big)\frac{dv}{v}
=|\,\beta\,|^{2-2s}C(\widetilde{\chi})
\int^\infty_{\frac{1}{|\beta|}} 
v^{2-s} F(v,\overline{\widetilde{\chi}})\frac{dv}{v}
\end{equation*}
by replacing  \,\(t=\frac{1}{|\,\beta\,|^2v}\)\, and \,\(\frac{dt}{t}=-\frac{dv}{v}\). 
In the case of \,\(s=1\), 
we have 
\begin{equation*}
\pi^{-1}L(1,\widetilde{\chi})=C(\widetilde{\chi})
\int_{\frac{1}{|\beta|}}^\infty F(t,\overline{\widetilde{\chi}})dt
+\int_{\frac1{|\beta|}}^\infty F(t,\widetilde{\chi})dt. 
\end{equation*}
We put 
\begin{equation*}
\smash{%
L(s,\widetilde{\chi})=
\sum_{m=1}^\infty \frac{b_m}{m^s}\quad (b_m\in \mathcal{O}). }
\end{equation*}
Then we have 
\begin{equation*}
L(s,\overline{\widetilde{\chi}})=\sum_{m=1}^\infty \frac{\overline{b_m}}{m^s}, \ \ \ 
F(t,\widetilde{\chi})=\sum_{m=1}^\infty b_m e^{-\pi mt},\ \ 
F(t,\overline{\widetilde{\chi}})=\sum_{m=1}^\infty \overline{b_m} e^{-\pi mt}. 
\end{equation*}
It follows from \(|\,\overline{b_m}\,|=|\,b_m\,|\) that 
\begin{equation*}
|\,F(t,\widetilde{\chi})\,|\le \sum_{m=1}^\infty |\,b_m\,| e^{-\pi mt},\quad 
|\,F(t,\overline{\widetilde{\chi}})\,|\le \sum_{m=1}^\infty |\,b_m\,| e^{-\pi mt}. 
\end{equation*}
As \,\(|\,C(\widetilde{\chi})\,|=1\), we see
\begin{equation*}
\pi^{-1}|\,L(1,\widetilde{\chi})\,|\le 2
\int_{\frac1{|\beta|}}^\infty \sum_{m=1}^{\infty}|\,b_m\,| e^{-\pi m t}dt
= 2\sum_{m=1}^{\infty}\frac{|\,b_m\,|}{\pi m}e^{-\pi m /|\,\beta\,|}. 
\end{equation*}
Multiplying by \(\pi\) on both sides and by using \,\(|\,b_m\,|\le \sigma_0(m)\sqrt{m}\le 2m\), 
where \(\sigma_0(m)\) denotes the number of positive divisors of $m$ 
(cf. \cite{K}, p.\ 96, Problem 4 of p.\ 97), 
we have 
\begin{equation*}
\smash{
|\,L(1,\widetilde{\chi})\,|\le 4\sum_{m=1}^{\infty}e^{-\pi m /|\,\beta\,|}
=\frac{4e^{-\pi  /|\,\beta\,|}}{1-e^{-\pi  /|\,\beta\,|}}
=\frac{4}{e^{\pi  /|\,\beta\,|}-1}}
\end{equation*}
as desired. 
\end{proof}
\vskip 10pt
\section{Estimate of the coefficients of elliptic Gauss sums}\label{estimate}\ 
\vskip 5pt
\noindent
In this section we show (2) implies (1) in Theorem \ref{0212a}. 
At first we prove Lemma \ref{less_than_l_0} whose proof has been reserved. 
\vskip 5pt
\textit{Proof of Lemma}\,\ref{less_than_l_0}. \ 
Since \,\(1/\big(\exp({\pi}/{|\beta|})-1\big)<{|\beta|}/{\pi}\), 
we have \,\(L(1,\widetilde{\chi})<4\times{|\beta|}/{\pi}\).
For \(\ell\equiv 1\bmod{8}\)\, and the conductor \((\beta)=((1+\i)^3\lambda)\), we see
\begin{equation*}
4\cdot\,\frac{2\sqrt{2}\cdot|\lambda|}{\pi}
>L(1,\widetilde{\chi})
=\varpi\,\frac12\,|A_{\lambda}|\,|\lambda|^{-1}\,|\lambda|^{\frac34}
=\varpi\,\frac12\,|A_{\lambda}|\,|\lambda|^{-\frac14}
\end{equation*}
from Theorem \ref{egs_A_lambda}, (\ref{Hecke_L_egs}), and Lemma \ref{Koblitz_type}. 
So that we have
\,\(|A_{\lambda}|<({16\sqrt{2}}/{\pi\,\varpi})\,|\lambda|^{\frac54}\). 
The right hand side is smaller than \,\(\tfrac12\,\ell\)\, for \,\(\ell\geq 97\)\, 
because \,\(97^{\frac38}=5.55\cdots>\frac{32\sqrt{2}}{\pi\varpi}=5.49\cdots\). 
For \,\(\ell<97\), the inequality \,\(|A_{\lambda}|<\ell/2\)\, actually holds by the tables in \cite{A}. 
\qed
\vskip 5pt
\begin{proof}[Proof of \,\((2)\Rightarrow(1)\) of Theorem {\rm\ref{0212a}}] 
Assume that \,\(\ell\,|\,G_{{d}}\). 
Then by Lemma \ref{0214d} we have \,\(\ell \,|\, \mathrm{egs}(\lambda)\). 
By Theorem 5.2, we have \,\(\tilde{\lambda}_0\,|\,A_\lambda\), 
where \,\(\widetilde{\lambda}_0\)\, is the prime defined in Theorem \ref{cong_for_1_mod_8}, and we see \,\(\ell\,|a_\lambda\). 
On the other hand, by Theorem \ref{egs_A_lambda} and Lemma \ref{less_than_l_0}, 
we have \,\(|\,a_\lambda\,|\le |\,A_\lambda\,|\le \ell/2\).
Thus we have \(a_\lambda=A_\lambda=0\) and \(\mathrm{egs}(\lambda)=0\). 
\end{proof}
\vskip 5pt

\end{document}